\documentclass[reqno]{amsart}
\usepackage{amsmath}
\usepackage{amssymb}
\usepackage{amsthm}
\usepackage{amsfonts}
\usepackage{amsrefs}
\numberwithin{equation}{section}
\usepackage{xcolor}
\usepackage[sc]{mathpazo}
\newtheorem{theorem}{Theorem}[section]
\newtheorem{proposition}[theorem]{Proposition}
\newtheorem{lemma}[theorem]{Lemma}
\newtheorem{corollary}[theorem]{Corollary}
\newtheorem*{thm1}{Theorem 1}
\newtheorem*{thm2}{Theorem 2}
\newtheorem*{thm3}{Theorem 3}
\newtheorem*{cor1}{Corollary 1}
\newtheorem{remark}[theorem]{Remark}

\begin{document}                                                 
	\title[Analysis of Minimizers of the Lawrence-Doniach Energy]{Analysis of Minimizers of the Lawrence-Doniach Energy for Superconductors in Applied Fields}                                 
	\author[P. Bauman]{Patricia Bauman} 
                           
	\address{P. B.\\Department of Mathematics\\Purdue University\\West Lafayette\\IN 47907.}                                    
	\email{baumanp@purdue.edu.}                                      
	\author[G. Peng]{Guanying Peng}  
	\address{G. P.\\Department of Mathematics\\University of Arizona\\Tucson\\AZ 85721.} 
	\email{gypeng@math.arizona.edu.}
	\date{}                                       
	\thanks{The authors were supported in part by NSF Grant DMS-1109459.}                                     
	\keywords{Lawrence-Doniach energy, layered superconductor, Ginzburg-Landau energy, minimizer}                                   
	\subjclass[2010]{35J60, 35J65, 35Q40}                                  
	\begin{abstract}
	We analyze minimizers of the Lawrence-Doniach energy
	for layered superconductors \textcolor{black}{with Josephson constant $\lambda$ and Ginzburg-Landau \textcolor{black}{parameter} $1/\epsilon$} in a bounded generalized cylinder $D=\Omega\times[0,L]$ in $\mathbb{R}^3$,where $\Omega$ is a bounded simply connected Lipschitz domain in $\mathbb{R}^2$. \textcolor{black} {Our main result is that in an applied magnetic field $\vec{H}_{ex}=h_{ex}\vec{e}_{3}$ which is perpendicular to the layers with $\left|\ln\epsilon\right|\ll h_{ex}\ll\epsilon^{-2}$,
		the minimum Lawrence-Doniach energy is given by
		$\frac{|D|}{2}h_{ex}\ln\frac{1}{\epsilon\sqrt{h_{ex}}}(1+o_{\epsilon,s}(1))$ as $\epsilon$ and the interlayer distance $s$ tend to zero.
		We also prove estimates on the behavior of the order parameters, induced magnetic field, and vorticity in this regime.
		Finally, we observe that as a consequence of our results, the same asymptotic formula holds for the minimum anisotropic three-dimensional Ginzburg-Landau energy in $D$ with anisotropic parameter $\lambda$ and $o_{\epsilon,s}(1)$ replaced by $o_{\epsilon}(1)$.}
	\end{abstract}                
	\maketitle

\section{Introduction}

The Lawrence-Doniach model was formulated by Lawrence and Doniach in 1971 as a macroscopic model for layered superconductors. While the standard Ginzburg-Landau model has been well accepted as a macroscopic model for isotropic superconductors, it does not account for the anisotropy in three-dimensional high temperature superconducting materials. For these materials, depending on the nature of the anisotropy in the material, physicists have used the Lawrence-Doniach model (which treats the superconducting material as a stack of parallel superconducting layers with nonlinear Josephson coupling between them) or the three-dimensional anisotropic Ginzburg-Landau model (which is a slight modification of the standard three-dimensional Ginzburg-Landau model).  \textcolor{black}{(See \cite{Iye}.)}

The standard two-dimensional Ginzburg-Landau model (with energy given by \eqref{GLenergy}) has been intensively investigated.  In this case, an analysis of the behavior of energy minimizers and their vortex structure in an applied magnetic field $h_{ex}\vec{e}_3$ with modulus $h_{ex}$ in different regimes (e.g., $h_{ex}\approx c|\ln\epsilon|,$  $|\ln\epsilon|\ll h_{ex}\ll\epsilon^{-2}$,  or $h_{ex}\geq\frac{C}{\epsilon^2}$) is now well understood. \textcolor{black}{In particular, under appropriate assumptions on $c$ and $C$ above, it was shown that these regimes correspond to the appearance of vortices in superconductivity, the intermediate mixed phase, and the normal (nonsuperconducting) phase, respectively.  (See \cite{Giorgi-Phillips}, \textcolor{black}{\cite{JS}} and \textcolor{black}{\cite{SS1}-\cite{SS}}.)  These results have now been generalized to the standard three-dimensional Ginzburg-Landau model (with energy given by \eqref{AGL} and $\lambda=1$). (See \textcolor{black}{\cite{ABO}}, \cite{BJOS1}, \cite{BJOS2}, \cite{Giorgi-Phillips} and \cite{Ka}.) In the three-dimensional case, defects in energy minimizers are vortex filaments, and behavior of the order parameter near the defects is not as well understood as in two-dimensional superconductors.  (See \textcolor{black}{\cite{ABO}}, \cite{BJOS1}, \cite{BJOS2}, and \cite{Giorgi-Phillips}.)}

For the Lawrence-Doniach model (with energy given by (\ref{LD})), \textcolor{black}{a considerable amount of work has been done. (See \cite{ABB1}-\cite{ABS2}, \cite{BK}, \cite{Chapman} and \cite{Pe}.) In particular,} an analysis with $h_{ex}$ in the first two regimes \textcolor{black}{was done for the gauge-periodic problem in \cite{ABS}; in that case the superconductor \textcolor{black}{was} assumed to occupy all of $\mathbb{R}^3$ and the gauge invariant quantities were assumed to be periodic with respect to a given parallelepiped.  For the non-gauge periodic case,} in the last regime, $h_{ex}\geq\frac{C}{\epsilon^2}$, it was shown in \cite{BK} that if $C$ is \textcolor{black}{a sufficiently large constant (independent of $\epsilon$ and the interlayer distance),} all minimizers of the Lawrence-Doniach energy are in the normal (nonsuperconducting) phase, that is, the order parameters on the layers, $\{u_n\}_{n=0}^{N}$, are all identically equal to zero, and the induced magnetic field, $\nabla \times  \vec{A}$, is identically equal to the applied magnetic field. \textcolor{black}{Recently, the second author has proved $\Gamma$-convergence of the normalized Lawrence-Doniach energies, $\{(\ln \epsilon)^{-2}\mathcal{G}_{LD}^{\epsilon, s}(\{u_n\}_{n=0}^N, \vec{A})\}$ with $h_{ex}$ in the first regime, assuming $h_{ex}\approx c|\ln\epsilon|$, along with an additional scaling assumption on $s$ versus $\epsilon$, as $\epsilon$ tends to zero.  (See \cite{Pe}.)}
 
In this paper, we investigate \textcolor{black}{minimizers of the Lawrence-Doniach energy with
$h_{ex}$ in the second regime, \textcolor{black}{i.e., $|\ln\epsilon|\ll h_{ex}\ll\epsilon^{-2}$,} without gauge-periodicity assumptions or a relative size condition on $s$ versus $\epsilon$.}

The Lawrence-Doniach model describes a layered superconductor occupying a cylinder \textcolor{black}{$D=\Omega\times\textcolor{black}{[0,L]}$ with cross-section $\Omega$, height $L$, and $N+1$ layers of superconducting material at $\Omega\times\{ns\}$ for $n=0,1,\cdots,N$, where $N \geq 1$ and $s=\frac{L}{N}$.  We assume throughout this paper that
$\Omega$ is a bounded simply connected Lipschitz domain in $\mathbb{R}^2$ \textcolor{black}{and that $L$ and $\lambda$ are fixed positive constants.}
In an applied magnetic field} \textcolor{black}{$\vec{H}=h_{ex} \vec{e}_3$,} the Lawrence-Doniach energy is given by
\begin{equation}\label{LD}\begin{split}
\mathcal{G}_{LD}^{\epsilon, s}(\{u_n\}_{n=0}^N, \vec{A})&= s\sum^N_{n=0} \int_\Omega\left[\frac{1}{2}|\hat{\nabla}_{\hat{A}_{n}}u_n|^2+\frac{1}{4\epsilon^2}(1-|u_n|^2)^2\right]d\hat{x}\\
&+s\sum^{N-1}_{n=0}\int_\Omega\frac{1}{2\lambda^{2}s^2}|u_{n+1}-u_{n}e^{\imath \int_{ns}^{(n+1)s}A^{3}dx_{3}}|^{2}d\hat{x}\\
&+\frac{1}{2}\int_{\mathbb{R}^3}|\nabla\times{\vec{A}}-h_{ex}\vec{e}_3|^{2}dx
\end{split}\end{equation}
for $(\{u_n\}_{n=0}^N, \vec{A})$ such that
\begin{equation}\label{E}
\begin{cases}
&\{u_n\}_{n=0}^N\in [H^{1}(\Omega;\mathbb{C})]^{N+1} \text{ and} \\
& \vec{A}\in E:=\{\vec{C}\in H^{1}_{loc}(\mathbb{R}^3;\mathbb{R}^3):(\nabla\times\vec{C})-h_{ex}\vec{e}_{3}\in L^{2}(\mathbb{R}^3;\mathbb{R}^3)\}.
\end{cases}
\end{equation}
Here $\epsilon>0$ is the reciprocal of the Ginzburg-Landau parameter, and $\lambda$ is the Josephson penetration depth. The applied magnetic field $h_{ex}\vec{e}_3$ is assumed to satisfy $|\ln{\epsilon}|\ll h_{ex}\ll\epsilon^{-2}$ as $\epsilon\rightarrow 0$. The complex valued function $u_n$ defined in $\Omega$ is the order parameter for the $n$th layer and $|u_n(x_1,x_2)|^2$ is the density of superconducting electron pairs at each point $(x_1,x_2,ns)$ on the $n$th layer. For a minimizer of the Lawrence-Doniach energy \eqref{LD}, $|u_n(x_1,x_2)| > 0$ corresponds to a superconducting state at $(x_1,x_2,ns)$, whereas $|u_n(x_1,x_2)|=0$ corresponds to a normal (nonsuperconducting) state at $(x_1,x_2,ns)$, in which the density of superconducting electrons is zero. The vector field $\vec{A}=(A^1,A^2,A^3)$ defined on $\mathbb{R}^3$ is called the magnetic potential; its curl, $\nabla\times\vec{A}$, is the induced magnetic field. We let $x=(x_1,x_2,x_3)$, $\hat{\nabla}=(\partial_{1},\partial_{2})$, $\hat{x}=(x_1,x_2)$, $\hat{A}=(A^1,A^2)$ and $\hat{A}_n(\hat{x})=(A^{1}(\hat{x},ns),A^{2}(\hat{x},ns))$, the trace of $\hat{A}$ on the $n$th layer.  We set $\hat{\nabla}_{\hat{A}_{n}}u_n=\hat{\nabla}u_n-\imath\hat{A}_nu_n$ on $\Omega$. In the following, given two complex numbers $u$ and $v$, we let $(u,v)=\frac{1}{2}(\bar u v+u\bar v)=\Re(u\bar v)$, which is an inner product of $u=u_1+\imath u_2$ and $v=v_1+\imath v_2$ in $\mathbb{C}$ that agrees with the inner product of $(u_1,u_2)$ and $(v_1,v_2)$ in $\mathbb{R}^2$.

\textcolor{black}{We remark that since} $\vec{A}\in H^{1}_{loc}(\mathbb{R}^3;\mathbb{R}^3)$, it follows from the trace theorem and the Sobolev imbedding theorem that its trace $\hat{A}_n\in H^{\frac{1}{2}}_{loc}(\mathbb{R}^2;\mathbb{R}^2) \subset L^4_{loc}(\mathbb{R}^2;\mathbb{R}^2)$ and therefore the Lawrence-Doniach energy $\mathcal{G}_{LD}^{\epsilon, s}(\{u_n\}_{n=0}^N, \vec{A})$ is well-defined and finite. The existence of minimizers in $ [H^{1}(\Omega;\mathbb{C})]^{N+1} \times E$ was shown by Chapman, Du and Gunzburger in \cite{Chapman}.
Each minimizer of $\mathcal{G}_{LD}^{\epsilon, s}$ corresponds to a physically realistic state for the layered superconductor. The minimizer satisfies the Euler-Lagrange equations associated to the Lawrence-Doniach energy. This system of equations is called the Lawrence-Doniach system and it is given by
\begin{equation*}
\begin{cases}
(\hat{\nabla}-\imath\hat{A}_n)^{2}u_n+\frac{1}{\epsilon^2}(1-|u_n|^2)u_n+P_n=0 & \text{ on } \Omega,\\
\nabla\times(\nabla\times\vec{A})=(j_1,j_2,j_3) & \text{ in } \mathbb{R}^3,\\
(\hat{\nabla}-\imath\hat{A}_n)u_n\cdot\vec{n}=0 & \text{ on } \partial\Omega,\\
\nabla\times\vec{A}-h_{ex}\vec{e}_3\in L^2(\mathbb{R}^3;\mathbb{R}^3)
\end{cases}
\end{equation*}
for all $n=0,1,\dddot\ ,N$, where
\begin{equation*}
P_n=\begin{cases}
\frac{1}{\lambda^{2}s^2}(u_{1}{\Upsilon}_{0}^{1}-u_0) & \text{if $n=0$,}\\
\frac{1}{\lambda^{2}s^2}(u_{n+1}{\Upsilon}_{n}^{n+1}+u_{n-1}\Upsilon_{n-1}^n-2u_n) & \text{if $0<n<N$,}\\
\frac{1}{\lambda^{2}s^2}(u_{N-1}\Upsilon_{N-1}^{N}-u_N) & \text{if $n=N$,}
\end{cases}
\end{equation*}
\begin{equation*}
\Upsilon_{n}^{n+1}=e^{\imath\int_{ns}^{(n+1)s}A^{3}dx_3} \text{ for $n=0,1,\dddot\ ,N-1$,}
\end{equation*}
\begin{equation*}
j_i=-s\sum\limits_{n=0}^{N}(\partial_{i}u_{n}-\imath A_n^iu_n,-\imath u_n)\chi_{\Omega}(x_1,x_2)dx_{1}dx_{2}\delta_{ns}(x_3) \text{ for $i=1,2,$}
\end{equation*}
\begin{equation*}
j_3=s\sum\limits_{n=0}^{N-1}\dfrac{1}{\lambda^{2}s^2}(u_{n+1}-u_{n}\Upsilon_{n}^{n+1},\imath u_{n}\Upsilon_{n}^{n+1})\chi_{\Omega}(x_1,x_2)\chi_{[ns,(n+1)s]}(x_3).
\end{equation*}
It was proved in \cite{BK} that a minimizer $(\{u_n\}_{n=0}^N, \vec{A})$ of (\ref{LD})
satisfies $|u_n|\leq 1$ a.e. in $\Omega$ for all $n=0,1,\dddot\ ,N$.

Two configurations $(\{u_n\}_{n=0}^N, \vec{A})$ and $(\{v_n\}_{n=0}^N, \vec{B})$ in $[H^{1}(\Omega;\mathbb{C})]^{N+1}\times E$ are called gauge equivalent if there exists a function $g\in H^{2}_{loc}(\mathbb{R}^3)$ such that
\begin{equation}\label{gauge}
\begin{cases}
u_n(\hat{x})=v_n(\hat{x})e^{\imath g(\hat{x},ns)} &\text{ in } \Omega,\\
\vec{A}=\vec{B}+\nabla g &\text{ in } \mathbb{R}^3.\\
\end{cases}
\end{equation}
Simple calculations show that $\mathcal{G}_{LD}^{\epsilon, s}$ (and each term in $\mathcal{G}_{LD}^{\epsilon, s}$) is invariant under the above gauge transformation, i.e., for two configurations $(\{u_n\}_{n=0}^N, \vec{A})$ and $(\{v_n\}_{n=0}^N, \vec{B})$ that are related by \eqref{gauge}, we have $\mathcal{G}_{LD}^{\epsilon, s}(\{u_n\}_{n=0}^N, \vec{A})=\mathcal{G}_{LD}^{\epsilon, s}(\{v_n\}_{n=0}^N, \vec{B})$. Let $\vec{a}=\vec{a}(x)$ be any fixed smooth vector field on $\mathbb{R}^3$ such that $\nabla\times\vec{a}=\vec{e}_3$ in
$\mathbb{R}^3$. For example, we may choose $\vec{a}(x)=\frac{1}{2}(-x_2,x_1,0)$. It was also proved in \cite{BK} that every pair $(\{u_n\}_{n=0}^N, \vec{A})\in[H^1(\Omega;\mathbb{C})]^{N+1}\times E$ is gauge equivalent to another pair $(\{v_n\}_{n=0}^N, \vec{B}) \in [H^1(\Omega;\mathbb{C})]^{N+1}\times K$ where
\begin{equation}\label{K}
K:=\{\vec{C}\in E:\nabla\cdot\vec{C}=0 \text{ and } \vec{C}-h_{ex}\vec{a}\in \check{H}^1(\mathbb{R}^3)\cap L^6(\mathbb{R}^3;\mathbb{R}^3)\}.
\end{equation}
Here the space $\check{H}^1(\mathbb{R}^3)$ represents the completion of $C_0^{\infty}(\mathbb{R}^3;\mathbb{R}^3)$ with respect to the seminorm
\begin{equation*}
\lVert\vec{C}\rVert_{\check{H}^1(\mathbb{R}^3)}=(\int_{\mathbb{R}^3}|\nabla\vec{C}|^2dx)^{\frac{1}{2}}.
\end{equation*}
In particular, any minimizer of $\mathcal{G}_{LD}^{\epsilon,s}$ in the admissible space $[H^1(\Omega;\mathbb{C})]^{N+1}\times E$ is gauge-equivalent to a minimizer in the space $[H^1(\Omega;\mathbb{C})]^{N+1}\times K$, called the ``Coulomb gauge" for $\mathcal{G}_{LD}^{\epsilon,s}$. It was shown in \cite{BK} that minimizers in the Coulomb gauge satisfy $u_n \in C^{\infty}(\Omega)$ and $\hat{A}_n \in H^1_{\text loc}(\mathbb{R}^2;\mathbb{R}^2)$ for all $n=0,1,\dddot\ ,N$. Throughout this paper, we take $\vec{a}(x)=\frac{1}{2}(-x_2,x_1,0)$.

Given the above definitions, our main results are the following:
\begin{thm1}\label{T1}
Assume $|\ln\epsilon|\ll h_{ex}\ll\epsilon^{-2}$ as $\epsilon\rightarrow 0$. Let $(\{u_n\}_{n=0}^N, \vec{A})\in[H^1(\Omega;\mathbb{C})]^{N+1}\times K$ be a minimizer of $\mathcal{G}_{LD}^{\epsilon, s}$. Then denoting the volume of $D$ by $|D|$, we have
\begin{equation*}
\begin{split}
\left| \mathcal{G}_{LD}^{\epsilon, s}(\{u_n\}_{n=0}^N, \vec{A})-\frac{|D|}{2}h_{ex}\ln\frac{1}{\epsilon\sqrt{h_{ex}}}\right| &\leq(Cs^\frac{1}{7}+o_{\epsilon}(1))\frac{|D|}{2}h_{ex}\ln\frac{1}{\epsilon\sqrt{h_{ex}}}\\
&=o_{\epsilon,s}(1)\frac{|D|}{2}h_{ex}\ln\frac{1}{\epsilon\sqrt{h_{ex}}}
\end{split}
\end{equation*}
\textcolor{black}{for all $\epsilon$ and $s$ sufficiently small, where $C$ is a positive constant depending only on the diameter of $\Omega$ and L.} In particular,
\begin{equation*}
\lim_{(\epsilon,s)\rightarrow(0,0)}\frac{\mathcal{G}_{LD}^{\epsilon,s}(\{u_n\}_{n=0}^N, \vec{A})}{h_{ex}\ln\frac{1}{\epsilon\sqrt{h_{ex}}}}=\frac{|D|}{2}.
\end{equation*}
\end{thm1}
\noindent (See Theorem 3.1 and Theorem 5.1.)

Here $o_{\epsilon}(1)$ denotes a quantity that converges to $0$ as $\epsilon\rightarrow 0$ and $o_{\epsilon,s}(1)$ denotes a quantity that converges to $0$ as $(\epsilon,s)\rightarrow(0,0)$. Theorem 1 generalizes a result in the gauge periodic case studied by Alama, Bronsard and Sandier for the energy (\ref{LD}) in which the domain $\Omega$ is replaced by a parallelogram $P$ in $\mathbb{R}^2$, the integral of $|\nabla\times{\vec{A}}-h_{ex}\vec{e}_3|^{2}$ is taken over $P\times[0,L]$ instead of over $\mathbb{R}^3$, and the minimization takes place among gauge periodic configurations $(\{u_n\}_{n=0}^N, \vec{A})$ in $\mathbb{R}^3$ with period $P\times[0,L]$. (See \cite{ABS}.)  In that case, they further showed that for a minimizer of the gauge periodic problem, the order parameters $u_n$ are all equal and $A^3$ is identically zero. In particular, the Josephson coupling term
\begin{equation}\label{JC}
s\sum^{N-1}_{n=0}\int_\Omega\frac{1}{2\lambda^{2}s^2}|u_{n+1}-u_{n}e^{\imath \int_{ns}^{(n+1)s}A^{3}dx_{3}}|^{2}d\hat{x}
\end{equation}
vanishes in that case. They also proved that $\hat{A}(\hat{x},\cdot)$ is periodic in $x_3$ with period $s$ and established certain symmetries between the layers in $\hat{A}$.

The results of Alama, Bronsard and Sandier indicated a close connection in the gauge periodic case between the Lawrence-Doniach energy and the two-dimensional Ginzburg-Landau energy $GL_\epsilon$ given by
\begin{equation}\label{GLenergy}
GL_\epsilon(u,\hat{A})=\frac{1}{2}\int_{\Omega}\left[|\hat{\nabla}_{\hat{A}}u|^2+\frac{1}{2\epsilon^2}(1-|u|^2)^2\right]d\hat{x}+\frac{1}{2}\int_{\mathbb{R}^2}(\hat{\text{curl}}\hat{A}-h_{ex})^2d\hat{x}
\end{equation}
for $h_{ex}$ as assumed above, where $\hat{\text{curl}}$ denotes the two-dimensional curl defined by $\hat{\text{curl}}(B^1,B^2)=\partial_{1}B^2-\partial_{2}B^1$. We remark that for a minimizer of the two-dimensional energy $GL_\epsilon$, the magnetic potential $\hat{A}$ satisfies $\hat{\text{curl}}\hat{A}=h_{ex}$ in $\mathbb{R}^2\setminus\Omega$.  (See Lemma 2.1 in \cite{Giorgi-Phillips}.) Therefore the minimum of $GL_\epsilon$ is equal to the minimum of $F_{\epsilon}$ given by
\begin{equation*}
F_{\epsilon}(u,\hat{A})=\frac{1}{2}\int_{\Omega}\left[|\hat{\nabla}_{\hat{A}}u|^2+\frac{1}{2\epsilon^2}(1-|u|^2)^2\right]d\hat{x}+\frac{1}{2}\int_{\Omega}(\hat{\text{curl}}\hat{A}-h_{ex})^2d\hat{x}.
\end{equation*}
(See Prop. 3.4 in \cite{SS} for bounded simply connected smooth domains and Prop. 2.1 in this paper for bounded simply connected Lipschitz domains.)

\textcolor{black}{The upper bound on the minimum Lawrence-Doniach energy in Theorem 1 is obtained by constructing a test function that is an extension of $N+1$ copies in each layer, $\Omega \times \{ns\}$, of a two-dimensional configuration $(u,(\textcolor{black}{A^1,A^2})(\hat x,ns))$. This construction is modeled after one that was used by Kachmar \cite{Ka} to prove an upper bound for the minimum value of the standard three-dimensional Ginzburg-Landau energy \textcolor{black}{with $h_{ex}$ in the same regime as in Theorem \ref{T1}}.}


A matching lower bound is much more difficult to establish. To obtain it, we prove that for a minimizer $(\{u_n\}_{n=0}^N,\vec{A})\in[H^1(\Omega;\mathbb{C})]^{N+1}\times K$  of $\mathcal{G}_{LD}^{\epsilon,s}$ \textcolor{black}{and any $h_{ex} > 0$ (with no assumption on $h_{ex}$ versus $\epsilon)$}, we have
\begin{equation}\label{Trace}
\frac{1}{2}\sum^{N-1}_{n=0}\int_{ns}^{(n+1)s}\int_{\Omega}|\hat{\text{curl}}\hat{A}(\hat{x},x_3)-
\hat{\text{curl}}\hat{A}_n(\hat{x})|^2d\hat{x}dx_3\leq Cs^{\frac {2}{7}}\mathcal{G}_{LD}^{\epsilon,s}(\{u_n\}_{n=0}^N,\vec{A})
\end{equation}
\textcolor{black}{for some constant C depending only on the diameter of $\Omega$ and $L$.  (See Theorem
\ref{TraceEstimate}.)  For ease of notation, let $M_{\epsilon}:=\frac{|D|}{2}h_{ex}\ln\frac{1}{\epsilon\sqrt{h_{ex}}}.$}
\textcolor{black}{Note that for $h_{ex}$ as in Theorem 1, $\epsilon^2 M_{\epsilon} \to 0$ and $\epsilon M_{\epsilon} \to \infty$ as $\epsilon \to 0$.}
 Our proof of \textcolor{black}{Theorem \ref{TraceEstimate}} uses a single layer potential representation formula for $\vec{A}$ proved by Bauman and Ko in \cite{BK}
as well as a priori estimates for single layer potentials (see \cite{Fabes} and \cite{Verchota}) and harmonic functions. \textcolor{black}{Inequality \eqref{Trace}} plays a crucial role in \textcolor{black}{our} proof of the lower bound, as it implies that \textcolor{black}{for minimizers of the Lawrence-Doniach energy}, the three-dimensional integral $\frac{1}{2}\int_D|\hat{\text{curl}}\hat{A}(\hat x,x_3)-h_{ex}|^{2}dx$ can be approximated within $o_{\epsilon,s}(1) M_{\epsilon}$ by the sum of two-dimensional integrals, $\Sigma_{n=0}^{N-1} \frac{s} {2} \int_{\Omega} |\hat{\text{curl}}\hat{A}_n(\hat x)-h_{ex}|^2 d \hat x$. \textcolor{black}{As a result
we have that for a minimizer, the first term in the energy (\ref{LD}) plus the magnetic term $\frac{1}{2}\int_D|\hat{\text{curl}}\hat{A}-h_{ex}|^{2}dx$ can be approximated within $o_{\epsilon,s}(1) M_{\epsilon}$} by the sum of $sF_\epsilon(u_n,\hat{A}_n)$. Combining this with the lower bound for the minimum of $F_\epsilon$ proved by Sandier and Serfaty in \cite{SS}, we obtain the lower bound in Theorem 1.

As a consequence of the proof of Theorem 1, we conclude that the Josephson coupling term \eqref{JC} contributes a lower order energy to the total Lawrence-Doniach energy. In fact, we obtain


\begin{thm2}  Assume $|\ln\epsilon|\ll h_{ex}\ll\frac{1}{\epsilon^2}$ as $\epsilon\rightarrow 0$. Let $(\{u_n\}_{n=0}^N,\vec{A})\in[H^1(\Omega;\mathbb{C})]^{N+1}\times K$ be a minimizer of $\mathcal{G}_{LD}^{\epsilon,s}$. Then
\begin{equation*}
\begin{split}
&s\sum^{N-1}_{n=0}\int_\Omega\frac{1}{2\lambda^{2}s^2}|u_{n+1}-u_{n}e^{\imath \int_{ns}^{(n+1)s}A^{3}dx_{3}}|^{2}d\hat{x}+\frac{1}{2}\int_{\mathbb{R}^3\setminus D}|\nabla\times\vec{A}-h_{ex}\vec{e}_3|^{2}dx\\
&\qquad\qquad+\frac{1}{2}\int_{D}\left[(\frac{\partial A^3}{\partial x_2}-\frac{\partial A^2}{\partial x_3})^{2}+(\frac{\partial A^1}{\partial x_3}-\frac{\partial A^3}{\partial x_1})^{2}\right]dx\leq o_{\epsilon,s}(1)M_{\epsilon}
\end{split}
\end{equation*}
as $(\epsilon,s)\rightarrow(0,0).$
\end{thm2}


\noindent (See Section 5 for the proof.)

Theorems 1 and 2 and the estimate \eqref{Trace} imply a strong influence up to leading order of the two-dimensional energy $F_\epsilon$ on the minimal Lawrence-Doniach energy.  \textcolor{black}{More precisely, we prove in Corollary 5.2 that for a minimizer $(\{u_n\}_{n=0}^N, \vec{A})\in[H^1(\Omega;\mathbb{C})]^{N+1}\times K$ of the Lawrence-Doniach energy, we have}
$$\textcolor{black}{\mathcal{G}_{LD}^{\epsilon, s}(\{u_n\}_{n=0}^N, \vec{A}) =  \left[ \sum^{N-1}_{n=0}sF_\epsilon(u_n,\hat {A}_n)\right] + o_{\epsilon,s}(1)M_{\epsilon}.}$$
\noindent Another consequence of Theorems 1 and 2 is the following:

\begin{thm3}
Under the assumptions of Theorem 1, \textcolor{black}{we have
\begin{equation*}
\frac{1}{N+1}\sum\limits_{n=0}^{N}\int_{\Omega} (1-|u_n|^2)^2 d\hat{x} +
\textcolor{black}{\frac{1}{2}}\int_{\mathbb{R}^3} |\frac{\nabla\times\vec{A}}{h_{ex}}-\vec{e}_3|^2 dx
\rightarrow 0
\end{equation*}}
and
\begin{equation*}
\frac{1}{N+1}\sum_{n=0}^N\frac{\mu_n}{h_{ex}}\rightarrow d\hat{x}\ \text{in}\ H^{-1}(\Omega)
\end{equation*}
as $(\epsilon,s)\rightarrow(0,0)$, where \textcolor{black}{$d\hat{x}$ is the two-dimensional Lebesgue measure and} $\mu_n$ is the vorticity in the $n$th layer, defined by $$\mu_n:=\hat{\textnormal{curl}}(\imath u_n,\hat{\nabla}_{\hat{A}_n}u_n)+\hat{\textnormal{curl}}\hat{A}_n.$$
\end{thm3}
\noindent (See Theorem 5.3.)


The convergence of the average scaled vorticity in the layers to the Lebesgue measure generalizes a result for minimizers of the two-dimensional Ginzburg-Landau energy $F_{\epsilon}$ studied by Sandier and Serfaty.  (See Cor. 8.1 in \cite{SS}.)  They showed that for minimizers of $F_\epsilon$,
the scaled vorticity measure $\frac{\mu}{h_{ex}}$ converges to $d\hat{x}$ in $H^{-1}(\Omega)$ as $\epsilon\rightarrow 0$. The vorticity measure $\mu_n$ in each layer is a gauge-invariant version of the Jacobian determinant of $u_n$, and is \textcolor{black}{analogous} to the vorticity in fluids.  If $u_n$ is given in polar coordinates by $\rho_n e^{i\theta_n}$, then $\mu_n=\hat{\textnormal{curl}}\{\rho_n^2 ({\hat {\nabla}} \theta_n-\hat {A}_n)\}+\hat{\textnormal{curl}}{\hat{A}}_n$. The above theorem indicates that on average there are numerous vortices and they have an approximately uniform distribution. More detailed results on the nature, \textcolor{black}{location}, and number of vortices for minimizers of the Lawrence-Doniach energy is an interesting open problem to which the results of this paper should be relevant.  \textcolor{black}{It would also be interesting to analyze the asymptotic behavior of minimizers when the applied magnetic field is $h_{ex} \vec{e}$ for $h_{ex}$ in the regime considered here but with $\vec {e}$ a unit vector that is not perpendicular to the horizontal layers.}

Recall that another model for a large class of high-temperature anisotropic superconductors is the three-dimensional anisotropic Ginzburg-Landau model. In \textcolor{black}{that} model, a mass tensor with unequal principal values is introduced to account for the anisotropic structure in the superconductor. (See \cite{Chapman} and \cite{Iye} for more background information.)  For a given admissible function $(\psi,\vec{A})$ in $H^1(D;\mathbb{C}) \times E$, the anisotropic Ginzburg-Landau energy $\mathcal{G}_{AGL}^{\epsilon,\lambda}$ \textcolor{black}{in an applied magnetic field $\vec{H}=h_{ex}\vec{e}_3$} is given by
\begin{equation}\label{AGL}
\begin{split}
\mathcal{G}_{AGL}^{\epsilon,\lambda}(\psi,\vec{A})&=\frac{1}{2}\int_D\left[|\hat{\nabla}_{\hat{A}}\psi|^2+\frac{1}{\lambda^2}|(\frac{\partial}{\partial x_3}-\imath A^3)\psi|^2\right]dx\\
&+\int_D\frac{(1-|\psi|^2)^2}{4\epsilon^2}dx+\frac{1}{2}\int_{\mathbb{R}^3}|\nabla\times\vec{A}-h_{ex}\vec{e}_3|^2 dx.
\end{split}
\end{equation}
Here $\lambda$ is the same constant as in the Lawrence-Doniach energy. If $\lambda=1$, the anisotropic Ginzburg-Landau energy becomes the standard three-dimensional Ginzburg-Landau energy. \textcolor{black}{For convenience in comparing $\mathcal{G}_{AGL}^{\epsilon,\lambda}$ and
the Lawrence-Doniach energy for any given $\lambda > 0$, set $\mathcal{G}_{LD}^{\epsilon,s,\lambda}
=\mathcal{G}_{LD}^{\epsilon,s}$.} \textcolor{black}{A} connection between the Lawrence-Doniach energy $\mathcal{G}_{LD}^{\epsilon, s,\lambda}$ and the anisotropic Ginzburg-Landau energy $\mathcal{G}_{AGL}^{\epsilon,\lambda}$ when \textcolor{black}{$\epsilon$, $\lambda$, and $h_{ex}$} are fixed and $s$ tends to zero was studied in \textcolor{black}{\cite{BK} and \cite{Chapman}}. In particular, it was shown in \cite{Chapman} that under this assumption, a subsequence of minimizers of the Lawrence-Doniach energy \textcolor{black}{(after extending the order parameter by linear interpolation between the layers in $D$) forms} a minimizing sequence of the anisotropic Ginzburg-Landau energy. (\textcolor{black}{See Theorem 5.1 in \cite{Chapman}}.)

Our last result concerns \textcolor{black}{a comparison result and} the asymptotic behavior of the minimum values of the two energies \textcolor{black}{for $\lambda > 0$ fixed but arbitrary as both $\epsilon$ and $s$ tend to zero. \textcolor{black}{The \textcolor{black}{following estimate (\ref{thm3.2})} extends a previous result proved in
		\cite{Ka} for the standard three-dimensional Ginzburg-Landau energy (satisfying $\lambda=1$).
		Thus when $\lambda=1$, our proof gives an alternate method for proving their result.}
	(See also \textcolor{black}{\cite{ABO}}, \cite{BJOS1} and \cite{BJOS2} for other interesting results on the asymptotic behavior of minimizers for the standard three-dimensional Ginzburg energy.) Using Theorem 1 of this paper and Theorem 5.1 of \cite{Chapman}, we prove:}


\begin{cor1}\label{c1}
Assume $|\ln\epsilon|\ll h_{ex}\ll\epsilon^{-2}$ as $\epsilon\rightarrow 0$ and $\lambda>0$. Let $(\{u_n\}_{n=0}^N,\vec{A})\in[H^1(\Omega;\mathbb C)]^{N+1}\times E$ be a minimizer of $\mathcal{G}_{LD}^{\epsilon,s,\lambda}$ and let $(\zeta,\vec{B})\in H^1(D;\mathbb C)\times E$ be a minimizer of $\mathcal{G}_{AGL}^{\epsilon,\lambda}$. We have
\begin{equation}\label{thm3.1}
|\mathcal{G}_{AGL}^{\epsilon,\lambda}(\zeta,\vec{B})-\mathcal{G}_{LD}^{\epsilon, s,\lambda}(\{u_n\}_{n=0}^N, \vec{A})|\leq o_{\epsilon,s}(1)
\frac{|D|}{2}h_{ex}\ln\frac{1}{\epsilon\sqrt{h_{ex}}}
\end{equation}
as $(\epsilon,s)\rightarrow(0,0)$, \textcolor{black}{and hence}
\begin{equation}\label{thm3.2}
|\mathcal{G}_{AGL}^{\epsilon,\lambda}(\zeta,\vec{B})-\frac{|D|}{2}h_{ex}\ln\frac{1}{\epsilon\sqrt{h_{ex}}}| \leq o_{\epsilon}(1)\frac{|D|}{2}h_{ex}\ln\frac{1}{\epsilon\sqrt{h_{ex}}}
\end{equation}
as $\epsilon\rightarrow 0$.
\end{cor1}


\noindent (See \textcolor{black}{Lemma} 6.1 and Corollary 6.2.)


Our paper is organized as follows: In Section 2 we state some preliminary results concerning the single layer potential representation formulas for $\hat{A}$ and $\hat{A}_n$. In Section 3 we prove the upper bound on the minimal Lawrence-Doniach energy. In Section 4 we prove the a priori estimate for the magnetic potential $\vec A$ stated in \eqref{Trace}. \textcolor{black}{In Section 5 we use \eqref{Trace} to prove the lower bound on the minimal Lawrence-Doniach energy which completes the proof of Theorem 1; we also prove Theorems 2 and 3.  Finally in Section 6 we prove Corollary 1.}

\section{Preliminaries}

As noted in the introduction, the Lawrence-Doniach energy is invariant under the gauge transformation \eqref{gauge} and minimizers of $\mathcal{G}_{LD}^{\epsilon,s}$ are gauge-equivalent to a minimizer in the ``Coulomb gauge". It was proved in \cite{BK} that, for a minimizer $(\{u_n\}_{n=0}^N, \vec{A})$ of $\mathcal{G}_{LD}^{\epsilon,s}$ in the ``Coulomb gauge", the magnetic potential $\vec{A}$ has an explicit representation formula using single layer potentials. Recall the definition of the space $\check{H}^1(\mathbb{R}^3)$ in the introduction. From \cite{BK}, each $\vec{C}\in\check{H}^1(\mathbb{R}^3)$ has a representative in $L^6(\mathbb{R}^3;\mathbb{R}^3)$ such that
\begin{equation*}
\lVert\vec{C}\rVert_{L^6(\mathbb{R}^3;\mathbb{R}^3)}\leq 2\lVert\vec{C}\rVert_{\check{H}^1(\mathbb{R}^3)}
\end{equation*}
and
\begin{equation*}
\lVert\vec{C}\rVert_{\check{H}^1(\mathbb{R}^3)}^2=\int_{\mathbb{R}^3}(|\nabla\cdot\vec{C}|^2+|\nabla\times\vec{C}|^2)dx.
\end{equation*}

We remark that the Lawrence-Doniach energy $\mathcal{G}_{LD}^{\epsilon, s}$ considered here is different from that studied in \textcolor{black}{\cite{BK} and \cite{Chapman}}, via a simple rescaling in \textcolor{black}{the energy and the applied magnetic field $\vec{H}$.  The scaling we use here is the same as that used by Sandier and Serfaty in \cite{SS} to analyze minimizers of the two-dimensional Ginzburg-Landau energy as $\epsilon$ tends to zero.}

\textcolor{black}{More precisely, setting $\kappa=\frac{1}{\epsilon}$ and letting $\mathcal{G}^{\kappa,s,\lambda}_{ld}$ be the Lawrence-Doniach energy with applied magnetic field $\vec{H}$ studied in \cite{Chapman} with $\psi_{n}$ as the order parameter for the $n$th layer and $\vec{A}_{ld}$ as the magnetic potential for $\mathcal{G}^{\kappa,s,\lambda}_{ld}$, respectively, we have
\begin{equation}\label{LDRescaling}
\begin{cases}
\mathcal{G}_{LD}^{\epsilon,s,\lambda}(\{u_n\}_{n=0}^N, \vec{A})|_{\vec{H}=\tau \kappa \vec{e}_3}  =\frac{\kappa^2}{2}\mathcal{G}^{\kappa,s,\lambda}_{ld}(\{\psi_n\}_{n=0}^N, \vec{A}_{ld})|_{\vec{H}=\tau \vec{e}_3},\\
\text {for } u_{n}=\psi_{n} \text{ and }
\vec{A}=\kappa\vec{A}_{ld}
\end{cases}
\end{equation}
for any $\tau > 0$.  A similar rescaling holds for the anisotropic Ginzburg-Landau energy, i.e.,
\begin{equation}\label{AGLRescaling}
\begin{cases}
\mathcal{G}_{AGL}^{\epsilon,\lambda}(\psi,\vec{A})|_{\vec{H}=\tau \kappa \vec{e}_3}=\frac{\kappa^2}{2}
\mathcal{G}^{\kappa,\lambda}_{em}(\psi_{em},\vec{A}_{em})|_{\vec{H}=\tau \vec{e}_3},\\
\text {for } \psi=\psi_{em} \text{ and }
\vec{A}=\kappa\vec{A}_{em},
\end{cases}
\end{equation}
where $\mathcal{G}^{\kappa,\lambda}_{em}$ is the anisotropic Ginzburg-Landau (or effective mass) energy introduced in \cite{Chapman} with applied magnetic field $\vec{H}$, and $\psi_{em}$, $\vec{A}_{em}$ are the order parameter and the magnetic potential for $\mathcal{G}^{\kappa,\lambda}_{em}$, respectively.} The above formulas will be used in Section 6.

The analysis in \cite{BK} (after appropriate rescaling) applies here without any difficulty. In particular, we have representation formulas for $A^1, A^2, A_n^1$ and $A_n^2$ for a minimizer of the Lawrence-Doniach energy in the Coulomb gauge as in Lemma 3.1, Theorem 3.2 and Corollary 3.3 in \cite{BK}. To state these formulas, we first define the single layer potential for our setting. For each $k\in\{0,1,\dddot\ ,N\}$, and for a given function $g\in L^p(\Omega\times\{ks\})$ with $1<p<\infty$, we define the operator $S_k$ by
\begin{equation*}
[S_k(g)](x)=\int_{\Omega\times\{ks\}}\frac{c}{|x-Q|}g(Q)d\sigma(Q)
\end{equation*}
for $x$ in $\mathbb{R}^3\backslash [\overline\Omega \times \{ks\}]$ where $d\sigma$ denotes the surface measure on the plane and $c=-\frac{1}{4\pi}$. (See \cite{Fabes} and \cite{Verchota} for results on  layer potentials in smooth and Lipschitz domains, respectively.) Let $(\{u_n\}_{n=0}^N, \vec{A})\in[H^1(\Omega;\mathbb{C})]^{N+1}\times K$ be a minimizer of $\mathcal{G}_{LD}^{\epsilon, s}$.  Define
$h_k^i$ in $L^2(\Omega)$ by
\begin{equation*}
h_k^i(\hat{x})=s(\partial_{i}u_{k}-\imath A_k^iu_k,-\imath u_k)\chi_{\Omega}(\hat{x})
\end{equation*}
and define $g_k^i$ in $L^2(\Omega\times\{ks\})$ by
\begin{equation*}
g_k^i(x)=\chi_{\Omega\times\{ks\}}(x)h_k^i(\hat x)
\end{equation*}
for $i=1,2$. Then the single layer potential of $g_k^i$ is
\begin{equation}\label{Sk}
\begin{split}
[S_k(g_k^i)](x)&=\int_{\Omega\times\{ks\}}\frac{c}{|x-Q|}g_k^i(Q)d\sigma(Q)\\
&=\int_{\Omega}\frac{c}{|x-(\hat{y},ks)|}h_k^i(\hat{y})d\hat{y}.
\end{split}
\end{equation}
With the above definitions, from the formulas in \cite{BK}, we have
\begin{equation}\label{A}
A^i(x)-h_{ex}a^i(x)=\sum_{k=0}^{N}S_k(g_k^i)(x) \text{ in } L^2_{loc}(\mathbb{R}^3)
\end{equation}
and
\begin{equation}\label{An}
A_n^i(\hat x)-h_{ex}a_n^i(\hat x)=t_n^i(\hat{x},ns)+\sum_{\substack{k=0\\k\ne n}}^{N}S_k(g_k^i)(\hat{x},ns) \text{ a.e. in } \mathbb{R}^2
\end{equation}
for $i=1,2$, where $t_n^i(\hat{x},ns)$ is the trace of $[S_n(g_n^i)](x)$ on $\mathbb{R}^2\times\{ns\}$, which is given by
\begin{equation*}
\begin{split}
t_n^i(\hat{x},ns)&=\int_{\Omega\times\{ns\}}\frac{c}{|(\hat{x},ns)-Q|}g_n^i(Q)d\sigma(Q)\\
&=\int_{\Omega}\frac{c}{|\hat{x}-\hat{y}|}h_n^i(\hat{y})d\hat{y},
\end{split}
\end{equation*}
and $a_n^i(\hat x)=a^i(\hat x,ns)$ corresponds to the trace of $a^i$ in $\mathbb{R}^2\times\{ns\}$.

In order to state further properties that will be used later, we need some definitions and results
from \cite{BK} (based on the theory of single layer potentials in \cite{Fabes} and \cite{Verchota})
concerning nontangential limits and nontangential maximal functions. For fixed $R>0$ and $0<\theta<\pi/2$, let
\begin{equation*}
\Gamma:=\Gamma_{R,\theta}=\{x\in\mathbb{R}^3:|x|<R \text{ and } |x\cdot\vec{e}_3|>|x|cos\theta\}
\end{equation*}
be a cone nontangential to the plane $\{x_3=0\}$ with vertex at the origin. Define
\begin{equation*}
\Gamma^+=\{x\in\Gamma:x_3>0\} \text{ and } \Gamma^-=\{x\in\Gamma:x_3<0\}.
\end{equation*}
For $(\hat{x},ns)\in\mathbb{R}^2\times\{ns\}$, let
\begin{equation*}
\Gamma(\hat{x},ns)=\{y\in\mathbb{R}^3:y-(\hat{x},ns)\in\Gamma\}.
\end{equation*}
Similarly, define
\begin{equation*}
\Gamma^+(\hat{x},ns)=\{y\in\mathbb{R}^3:y-(\hat{x},ns)\in\Gamma^+\}
\end{equation*}
and
\begin{equation*}
\Gamma^-(\hat{x},ns)=\{y\in\mathbb{R}^3:y-(\hat{x},ns)\in\Gamma^-\}.
\end{equation*}
For a function $u$ defined in $\Gamma(\hat{x},ns)$, define the nontangential limit (n.t.limit) of $u(y)$ as $y\rightarrow(\hat{x},ns)$ by
\begin{equation*}
\underset{y\rightarrow(\hat{x},ns)}{\text{n.t.limit }}u(y)=\lim\limits_{y\rightarrow(\hat{x},ns)}\{u(y):y\in\Gamma(\hat{x},ns)\},
\end{equation*}
provided the limit exists. Also we have the following definition of the nontangential maximal function of $u$ at $(\hat{x},ns)$, denoted by $u^*(\hat{x},ns)=u^*_{R,\theta}(\hat x,ns)$ for each $n$ in $\{0,1, \cdots, N\}$.
\begin{equation*}
u^*(\hat{x},ns)=\sup\{|u(y)|:y\in\Gamma_{R,\theta}(\hat{x},ns)\}.
\end{equation*}
By Theorem 3.2 in \cite{BK}, $S_n(g_n^i)\in W_{loc}^{1,2}(\mathbb{R}^3)\cap C^{\infty}(\mathbb{R}^3\setminus\overline{\Omega}_n)$ and $t_n^i(\hat x,ns)\in W_{loc}^{1,2}(\mathbb{R}^2\times\{ns\})$.

Throughout this paper, we let $\theta=\frac {\pi} {4}$ and $R=1$.  Also, let $R_0$ be a fixed constant satisfying
\begin{equation}\label{r0}
R_0 \geq 2(\text{diam }\Omega),
\end{equation}
where $\text{diam } \Omega$ is the diameter of $\Omega$.  It follows that the nontangential maximal functions of $S_n(g_n^i)$ and $\nabla S_n(g_n^i)$ are in $L^2_{loc}(\mathbb{R}^2\times\{ns\})$ and
\begin{equation}\label{nontangential_maximal}
\lVert(S_n(g_n^i))^*\rVert_{L^2(\Omega\times\{ns\})} + \lVert(\nabla S_n(g_n^i))^*\rVert_{L^2(\Omega\times\{ns\})}\leq C\lVert g_n^i\rVert_{L^2(\Omega\times\{ns\})}
\end{equation}
for some constant $C$ depending only on $R_0$. (This property of the constant $C$ uses the fact that $\Omega$ is a subset of a disk of radius $R_0$ in $\mathbb{R}^2$. See \cite{BK}.)
Also $t_n^i(\hat x,ns)$ and $\hat{\nabla}t_n^i(\hat x,ns)$ are the nontangential limits of $S_n(g_n^i)$ and $\hat{\nabla}S_n(g_n^i)$, respectively, pointwise a.e. in $\mathbb{R}^2\times\{ns\}$ and in $L^2_{loc}(\mathbb{R}^2\times\{ns\})$, and we have
\begin{equation*}
\nabla S_n(g_n^i)(x)=\int_{\mathbb{R}^2}\frac{-c(x-(\hat{y},ns))}{|x-(\hat{y},ns)|^3}h_n^i(\hat{y})d\hat{y} \text{ a.e. in } \mathbb{R}^3
\end{equation*}
and
\begin{equation*}
(\hat{\nabla}t_n^i)(\hat{x},ns)=P.V.\int_{\mathbb{R}^2}\frac{-c(\hat{x}-\hat{y})}{|\hat{x}-\hat{y}|^3}h_n^i(\hat{y})d\hat{y} \text{ a.e. in } \mathbb{R}^2\times\{ns\},
\end{equation*}
where $P.V.$ denotes the principal-valued integral. In addition, we have
\begin{equation}\label{grad_tn}
\lVert t_n^i\rVert_{L^2(\Omega)}+\lVert\hat{\nabla}t_n^i\rVert_{L^2(\Omega)}\leq C\lVert h_n^i\rVert_{L^2(\Omega)}
\end{equation}
for some constant $C$ depending only on $R_0$.

The above representation formulas and properties of the single layer potential will be used in Section 4 in our proof of the lower bound for the minimum Lawrence-Diniach energy.

We conclude this section with the following proposition concerning the minimum of the energies $GL_{\epsilon}$ and $F_{\epsilon}$ over bounded simply connected Lipschitz domains, which is a modification of Proposition 3.4 in \cite{SS}.

\begin{proposition}
Assume that $\Omega\subset\mathbb{R}^2$ is a bounded simply connected Lipschitz domain. Let
\begin{equation*}
X=\{(v,b)\in H^1(\Omega;\mathbb{C})\times H^1_{loc}(\mathbb{R}^2;\mathbb{R}^2):(\textnormal{curl }b-h_{ex})\in L^2(\mathbb{R}^2)\}
\end{equation*}
and
\begin{equation*}
X_{\Omega}=\{(v,b)\in H^1(\Omega;\mathbb{C})\times H^1(\Omega;\mathbb{R}^2)\}.
\end{equation*}
Then we have
\begin{equation*}
\min_{(v,b)\in X}GL_{\epsilon}(v,b)=\min_{(v,b)\in X_{\Omega}}F_{\epsilon}(v,b).
\end{equation*}
\end{proposition}
\begin{proof}
Note that for any function $(v,b)$ in $X$, we have $(v,b|_{\Omega}) \in X_{\Omega}$. From this and the definitions of $GL_{\epsilon}$ and $F_{\epsilon}$, we obtain
$$\min_{(v,b)\in X}GL_{\epsilon}(v,b)\geq\min_{(v,b)\in X_{\Omega}}F_{\epsilon}(v,b).$$
Given a minimizer $(v,b)\in X_{\Omega}$ of $F_{\epsilon}$, let $\phi$ solve $\Delta\phi=H$, where
$H=(\text{curl }b-h_{ex})\cdot\chi_{\Omega}\in L^2({\mathbb{R}^2})$. We may take $\phi$ to be the Newtonian potential of $H$. By standard estimates for Newtonian potentials, we have $\phi\in H^2_{loc}(\mathbb{R}^2)$. Define $\tilde{b}=(-\partial_2\phi,\partial_1\phi)+(0,h_{ex}x_1)\in H^1_{loc}(\mathbb{R}^2;\mathbb{R}^2)$. Direct calculations show that $\text{curl }(\tilde b|_{\Omega})=\text{curl }b$ in $\Omega$, where $\tilde b|_{\Omega}$ is the restriction of $\tilde b$ on $\Omega$. Therefore there exists $f\in H^2(\Omega)$ such that $\tilde{b}|_{\Omega}=b+\nabla f$ in $\Omega$. Define $\tilde v=ve^{if}$. Simple calculations using the formula above for $H$ and gauge invariance imply that

$$GL_{\epsilon}(\tilde v,\tilde{b})=F_{\epsilon}(\tilde v,\tilde{b}|_{\Omega})=F_{\epsilon}(v,b)=\min_{(v,b)\in X_{\Omega}}F_{\epsilon}(v,b).$$
Hence $\min_{(v,b)\in X} GL_{\epsilon}(v,b) \leq \min_{(v,b)\in X_{\Omega}} F_{\epsilon}(v,b)$ and equality must hold.
\end{proof}

\section{Upper bound}

From here on in the paper, we let $C_0$ denote any constant that is independent of $\epsilon$, $s$, $\Omega$, $L$, $D$, $\lambda$, and $R_0$ for all $\epsilon$ and $s$ sufficiently small.  Recall that in our notation, $\hat{V}$ denotes a two-dimensional vector $V=(V_1,V_2)$ and  $\vec{W}$ denotes a three-dimensional vector $W=(W_1,W_2,W_3)$.  Also $\hat {\text{curl }} {\hat V}=\partial_{1}V_2-\partial_{2} V_1$.  We will denote by $\hat{\nabla}$ the operator $(\partial_{1},\partial_{2})$.

In this section we prove the upper bound on the minimum Lawrence-Doniach energy. The main idea in the proof is to construct a test configuration with the induced magnetic field $\nabla\times\vec A$ identically equal to the applied magnetic field $h_{ex}\vec e_3$, such that each $u_n$ is a rescaling of a function which minimizes a simplified two-dimensional Ginzburg-Landau energy introduced by Kachmar \cite{Ka}. For such a test configuration, its Lawrence-Doniach energy becomes $N+1$ identical copies of a two-dimensional Ginzburg-Landau energy that was studied in detail in \cite{Ka}. Using the sharp upper bound estimate for this two-dimensional Ginzburg-Landau energy, we are led to the sharp upper bound for the minimum Lawrence-Doniach energy.

\begin{theorem}\label{UB}
Assume $|\ln\epsilon|\ll h_{ex}\ll\frac{1}{\epsilon^2}$ as $\epsilon\rightarrow 0$. Let $(\{u_n\}_{n=0}^N,\vec{A})$ be a minimizer of $\mathcal{G}_{LD}^{\epsilon,s}$. Then we have
\begin{equation*}
\mathcal{G}_{LD}^{\epsilon,s}(\{u_n\}_{n=0}^N,\vec{A})\leq \frac{|D|}{2}h_{ex} (\ln\frac{1}{\epsilon\sqrt{h_{ex}}})\bigl(1+c(\epsilon,s)\bigr)
\end{equation*}
\textcolor{black}{for all $\epsilon$ sufficiently small}, where
\begin{equation*}
c(\epsilon,s)=\frac{s}{L}+o_{\epsilon}(1)=o_{\epsilon,s}(1).
\end{equation*}
\end{theorem}

\begin{proof}
First we introduce some \textcolor{black}{notation} defined in \cite{Ka}. Let $K=(-\frac{1}{2},\frac{1}{2})\times(-\frac{1}{2},\frac{1}{2})$.  We define the following simplified two-dimensional Ginzburg-Landau energy
\begin{equation*}
E^{2D}(u)=\frac{1}{2}\int_{K}\left[|(\hat{\nabla}-\imath h_{ex}\hat{a})u|^2+\frac{1}{2\epsilon^2}(1-|u|^2)^2\right]d\hat x,
\end{equation*}
where $\hat a(\hat x)=\frac{1}{2}(-x_2,x_1)$, and the space
\begin{equation*}
E_{h_{ex}}:=\{u\in H^1_{loc}(\mathbb{R}^2;\mathbb{C}):u(x_1+1,x_2)=e^{\imath h_{ex}\frac{x_2}{2}}u(x_1,x_2) \text{ and } u(x_1,x_2+1)=e^{-\imath h_{ex}\frac{x_1}{2}}u(x_1,x_2)\}.
\end{equation*}
Let
\begin{equation*}
m_{p}(h_{ex},\epsilon)=\inf\{E^{2D}(u):u\in E_{h_{ex}}\}.
\end{equation*}
By Theorem 2.2 in \cite{Ka}, for $h_{ex}$ satisfying $|\ln\epsilon|\ll h_{ex}\ll\frac{1}{\epsilon^2}$, we have
\begin{equation}\label{3.1.1}
m_{p}(h_{ex},\epsilon)=\frac{h_{ex}}{2}\ln\frac{1}{\epsilon\sqrt{h_{ex}}}\bigl(1+o_{\epsilon}(1)\bigr)
\end{equation}
as $\epsilon\rightarrow 0$.

Let $\tilde{u}\in E_{h_{ex}}$ be a minimizer of $E^{2D}$. Let $b=\epsilon^2 h_{ex}$ and $l=(\frac{h_{ex}}{|\ln\epsilon|})^{\frac{1}{4}}\frac{1}{\sqrt{h_{ex}}}$. It is easy to check, using the assumption $|\ln\epsilon|\ll h_{ex}\ll\frac{1}{\epsilon^2}$, that $b\ll 1$, $l\ll 1$ and $l\sqrt{h_{ex}}\gg 1$. Define $\tilde h_{ex}=\frac{1}{l^2}$ and $\tilde\epsilon=\sqrt{b}l$. It follows that $|\ln\tilde{\epsilon}|\ll \tilde h_{ex}\ll\frac{1}{\tilde\epsilon^2}$. Define $u(\hat x)=\tilde u(l\sqrt{h_{ex}}\hat x)$. We construct a test configuration $(\{v_n\}_{n=0}^N,\vec B)$ for the Lawrence-Doniach energy $\mathcal{G}_{LD}^{\epsilon,s}$ such that
\begin{equation*}
\begin{cases}
v_n(\hat x)=u(\hat x) \text{ for all } n=0,1,\dddot\ ,N,\\
\vec B(x)=\frac{h_{ex}}{2}(-x_2,x_1,0).
\end{cases}
\end{equation*}

Since $\nabla\times\vec B=h_{ex}\vec e_3$ and $B^3=0$, it is clear that
\begin{equation}\label{3.1.2}
\mathcal{G}_{LD}^{\epsilon,s}(\{v_n\}_{n=0}^N,\vec B)=(N+1)\frac{s}{2}\int_{\Omega}\left[|(\hat{\nabla}-\imath h_{ex}\hat{a})u|^2+\frac{1}{2\epsilon^2}(1-|u|^2)^2\right]d\hat x.
\end{equation}
Let $K_{\epsilon}=(-\frac{1}{2l\sqrt{h_{ex}}},\frac{1}{2l\sqrt{h_{ex}}})\times(-\frac{1}{2l\sqrt{h_{ex}}},\frac{1}{2l\sqrt{h_{ex}}})$. The change of variable $\hat y=l\sqrt{h_{ex}}\hat x$ yields
\begin{equation*}
\begin{split}
\frac{1}{2}\int_{K_{\epsilon}}&\left[|(\hat{\nabla}-\imath h_{ex}\hat{a}(\hat x))u(\hat x)|^2+\frac{1}{2\epsilon^2}(1-|u(\hat x)|^2)^2\right]d\hat x\\
&=\frac{1}{2}\int_{K}\left[|(\hat{\nabla}-\imath \tilde h_{ex}\hat{a}(\hat y))\tilde u(\hat y)|^2+\frac{1}{2\tilde\epsilon^2}(1-|\tilde u(\hat y)|^2)^2\right]d\hat y.
\end{split}
\end{equation*}
Using the fact that $\tilde{u}\in E_{h_{ex}}$ is a minimizer of $E^{2D}$ and \eqref{3.1.1}, we have
\begin{equation}\label{3.1.3}
\frac{1}{2}\int_{K_{\epsilon}}\left[|(\hat{\nabla}-\imath h_{ex}\hat{a}(\hat x))u(\hat x)|^2+\frac{1}{2\epsilon^2}(1-|u(\hat x)|^2)^2\right]d\hat x=\frac{\tilde h_{ex}}{2}\ln\frac{1}{\tilde\epsilon\sqrt{\tilde h_{ex}}}\bigl(1+o_{\tilde\epsilon}(1)\bigr)
\end{equation}
as $\epsilon\rightarrow 0$.

Let $L_{\epsilon}$ be the lattice in $\mathbb{R}^2$ generated by $K_{\epsilon}$. Let $\{K_i\}$ be the collection of squares formed by the lattice $L_{\epsilon}$ such that $K_i\cap\Omega\ne\emptyset$ and $\Omega\subset\bigcup_{i}K_i$. Simple calculations using the definitions of the space $E_{h_{ex}}$ and the function $u$ yield that
\begin{equation*}
\frac{1}{2}\int_{K_i}\left[|(\hat{\nabla}-\imath h_{ex}\hat{a})u|^2+\frac{1}{2\epsilon^2}(1-|u|^2)^2\right]d\hat x=\frac{1}{2}\int_{K_{\epsilon}}\left[|(\hat{\nabla}-\imath h_{ex}\hat{a})u|^2+\frac{1}{2\epsilon^2}(1-|u|^2)^2\right]d\hat x.
\end{equation*}
Therefore we have
\begin{equation}\label{3.1.5}
\begin{split}
&\frac{1}{2}\int_{\Omega}\left[|(\hat{\nabla}-\imath h_{ex}\hat{a})u|^2+\frac{1}{2\epsilon^2}(1-|u|^2)^2\right]d\hat x\\
&\quad\quad\quad\quad\quad\quad\leq\sum\limits_{i}\frac{1}{2}\int_{K_i}\left[|(\hat{\nabla}-\imath h_{ex}\hat{a})u|^2+\frac{1}{2\epsilon^2}(1-|u|^2)^2\right]d\hat x\\
&\quad\quad\quad\quad\quad\quad=\frac{\sum\limits_{i}|K_i|}{|K_{\epsilon}|}\frac{1}{2}\int_{K_{\epsilon}}\left[|(\hat{\nabla}-\imath h_{ex}\hat{a})u|^2+\frac{1}{2\epsilon^2}(1-|u|^2)^2\right]d\hat x.
\end{split}
\end{equation}
Since $\sum\limits_{i}|K_i|=|\Omega|(1+o_{\epsilon}(1))$ as $\epsilon\rightarrow 0$, we deduce from \eqref{3.1.5} and \eqref{3.1.3} that
\begin{equation*}
\frac{1}{2}\int_{\Omega}\left[|(\hat{\nabla}-\imath h_{ex}\hat{a})u|^2+\frac{1}{2\epsilon^2}(1-|u|^2)^2\right]d\hat x\leq\frac{|\Omega|}{|K_{\epsilon}|}\frac{\tilde h_{ex}}{2}\ln\frac{1}{\tilde\epsilon\sqrt{\tilde h_{ex}}}\bigl(1+o_{\epsilon}(1)\bigr).
\end{equation*}
Using $|K_{\epsilon}|=\frac{1}{l^2h_{ex}}$, $\tilde h_{ex}=\frac{1}{l^2}$ and $\tilde\epsilon=\sqrt{b}l$, we conclude that
\begin{equation}\label{3.1.6}
\frac{1}{2}\int_{\Omega}\left[|(\hat{\nabla}-\imath h_{ex}\hat{a})u|^2+\frac{1}{2\epsilon^2}(1-|u|^2)^2\right]d\hat x\leq\frac{|\Omega|}{2}h_{ex}\ln\frac{1}{\epsilon\sqrt{h_{ex}}}\bigl(1+o_{\epsilon}(1)\bigr).
\end{equation}
Combining \eqref{3.1.2} with \eqref{3.1.6} we obtain
\begin{equation*}
\mathcal{G}_{LD}^{\epsilon,s}(\{v_n\}_{n=0}^N,\vec B)\leq\frac{|D|}{2}h_{ex}\ln\frac{1}{\epsilon\sqrt{h_{ex}}}\bigl(1+o_{\epsilon}(1)+\frac{s}{L}\bigr).
\end{equation*}
\end{proof}

\section{An a priori estimate}

In this section we \textcolor{black}{prove the inequality \eqref{Trace} for the magnetic potential $\vec A$.  (See Theorem 4.3.) Throughout this section, we assume only that $h_{ex}>0$, $\epsilon>0$ and $(\{u_n\}_{n=0}^N,\vec A)\in[H^1(\Omega;\mathbb{C})]^{N+1}\times K$ is a minimizer of the Lawrence-Doniach energy $\mathcal{G}_{LD}^{\epsilon,s}=\mathcal{G}_{LD}^{\epsilon,s,\lambda}$. Recall that
$\lambda > 0$ is assumed to be fixed but arbitrary throughout this paper.}

First note that since $\vec{a}(x)=\frac{1}{2}(-x_2,x_1,0)$ is independent of $x_3$, we have $$\hat{A}(\hat x,x_3)-\hat{A}_n(\hat x)=\left( \hat{A}(\hat x,x_3)-h_{ex}\hat{a}(\hat x,x_3)\right) -\left( \hat{A}_n(\hat x)-h_{ex}\hat{a}_n(\hat x)\right) $$
for $(\hat{x},x_3)$ in $\Omega\times[ns,(n+1)s)$. Therefore
\begin{equation*}
\begin{split}
\hat{\text{curl}}\hat{A}(\hat x,x_3)-\hat{\text{curl}}\hat{A}_n(\hat x)=&\hat{\text{curl}} (\hat{A}-h_{ex}\hat{a})(\hat{x},x_3)-\hat{\text{curl}}(\hat{A}_n-h_{ex}\hat{a}_n)(\hat{x}) \\
=&\frac{\partial}{\partial x_1}\left( (A^2-h_{ex}a^2)(\hat{x},x_3)-(A_n^2-h_{ex}a_n^2)(\hat{x})\right) \\
-&\frac{\partial}{\partial x_2}\left( (A^1-h_{ex}a^1)(\hat{x},x_3)-(A_n^1-h_{ex}a_n^1)(\hat{x})\right)
\end{split}
\end{equation*}
and it follows from this, \eqref{A} and \eqref{An} and the regularity results described in Section 2 that
\begin{equation}\label{4.1}
\begin{split}
&\hat{\text{curl}}\hat{A}(\hat x,x_3)-\hat{\text{curl}}\hat{A}_n(\hat x)\\
=&\left\lbrace \sum\limits_{\substack{k=0\\k\ne n}}^{N}\frac{\partial}{\partial x_1}(S_k(g_k^2)\left( \hat x,x_3)-S_k(g_k^2)(\hat x,ns)\right) +\frac{\partial}{\partial x_1}\left( S_n(g_n^2)(\hat x,x_3)-t_n^2(\hat x,ns)\right) \right\rbrace \\
-&\left\lbrace \sum\limits_{\substack{k=0\\k\ne n}}^{N}\frac{\partial}{\partial x_2}\left( S_k(g_k^1)(\hat x,x_3)-S_k(g_k^1)(\hat x,ns)\right) +\frac{\partial}{\partial x_2}\left( S_n(g_n^1)(\hat x,x_3)-t_n^1(\hat x,ns)\right) \right\rbrace
\end{split}
\end{equation}
in $L^2_{loc}(\mathbb{R}^3)$, where $h_k^i(\hat{x})=s(\partial_{i}u_{k}-\imath A_k^iu_k,-\imath u_k)\chi_{\Omega}(\hat{x})$ and $g_k^i(x)=\chi_{\Omega\times\{ks\}}(x)h_k^i(\hat x)$ for $i=1,2$.
\begin{lemma}
Assume that $h_{ex}>0$ and $\epsilon>0$ \textcolor{black} {(with no assumption on their relative values).}  Let $(\{u_n\}_{n=0}^N,\vec A)\in[H^1(\Omega;\mathbb{C})]^{N+1}\times K$ be a minimizer of $\mathcal{G}_{LD}^{\epsilon,s}$. We have
\begin{equation}\label{4.2}
\frac{1}{2}\sum\limits_{n=0}^{N-1}\int_{ns}^{(n+1)s}\int_{\Omega}|\hat{\mathrm{curl}}\hat{A}(\hat x,x_3)-\hat{\mathrm{curl}}\hat{A}_n(\hat x)|^2d\hat{x}dx_3\leq E_1+E_2,
\end{equation}
where
\begin{equation*}
\begin{split}
E_1=&\sum\limits_{n=0}^{N-1}\int_{ns}^{(n+1)s}\int_{\Omega} \bigl| \sum\limits_{\substack{k=0\\k\ne n}}^{N}\frac{\partial}{\partial x_1}\left(S_k(g_k^2)(\hat x,x_3)-S_k(g_k^2)(\hat x,ns)\right)\\
&\ \ \ \ \ \ \ \ \ \ \ \ \ \ \ \ \ \ \ \ +\frac{\partial}{\partial x_1}\left(S_n(g_n^2)(\hat x,x_3)-t_n^2(\hat x,ns)\right)\bigr|^2 d\hat{x}dx_3
\end{split}
\end{equation*}
and
\begin{equation*}
\begin{split}
E_2=&\sum\limits_{n=0}^{N-1}\int_{ns}^{(n+1)s}\int_{\Omega}\bigl|\sum\limits_{\substack{k=0\\k\ne n}}^{N}\frac{\partial}{\partial x_2}\left( S_k(g_k^1)(\hat x,x_3)-S_k(g_k^1)(\hat x,ns)\right) \\
&\ \ \ \ \ \ \ \ \ \ \ \ \ \ \ \ \ \ \ \ +\frac{\partial}{\partial x_2}\left( S_n(g_n^1)(\hat x,x_3)-t_n^1(\hat x,ns)\right) \bigr|^2d\hat{x}dx_3.
\end{split}
\end{equation*}
\end{lemma}
\begin{proof}
Since $|u_k|\leq 1$ a.e. in $\Omega$ (see \cite{BK}), we have
$$|h_k^i|\leq s|\partial_{i}u_{k}-\imath A_k^iu_k|$$
and
$$\lVert g_k^i\rVert_{L^2(\Omega\times\{ks\})}^2=\lVert h_k^i\rVert_{L^2(\Omega)}^2\leq s^2\lVert\partial_{i}u_{k}-\imath A_k^iu_k\rVert_{L^2(\Omega)}^2.$$
Applying the elementary inequality $(a-b)^2\leq 2a^2+2b^2$ to the representation formula for $\hat{\text{curl}}\hat{A}(\hat x,x_3)-\hat{\text{curl}}\hat{A}_n(\hat x)$ in \eqref{4.1} and taking the sum of the integrals, we obtain \eqref{4.2}.
\end{proof}

\begin{lemma}
Under the assumptions of Lemma 4.1, \textcolor{black}{there is a constant $C$ depending only on $R_0$ (\textcolor{black}{defined in (\ref{r0})}) and $L$ such that for all $s$ sufficiently small (independent of $\epsilon$),} we have
\begin{equation*}
E_1\leq(N+1)\sum\limits_{k=0}^{N}Cs^{\frac{2}{7}}\lVert h^2_k\rVert^2_{L^2(\Omega)}
\end{equation*}
and
\begin{equation*}
E_2\leq(N+1)\sum\limits_{k=0}^{N}Cs^{\frac{2}{7}}\lVert h^1_k\rVert^2_{L^2(\Omega)}.
\end{equation*}
\end{lemma}
\begin{proof}
In the following we analyze $E_1$ and the analysis for $E_2$ will be similar. First define
\begin{equation*}
\Delta_{n,k}(\hat x,x_3)=\frac{\partial}{\partial x_1}[S_k(g_k^2)(\hat x,x_3)-S_k(g_k^2)(\hat x,ns)]
\end{equation*}
for $n\ne k$ and
\begin{equation*}
\Delta_{n,n}(\hat x,x_3)=\frac{\partial}{\partial x_1}[S_n(g_n^2)(\hat x,x_3)-t_n^2(\hat x,ns)].
\end{equation*}
Note that for $n\ne k$, $\Delta_{n,k}$ is $C^{\infty}$ in $\mathbb{R}^3\setminus\overline{\Omega}_k$ since it is harmonic there. By the Cauchy-Schwartz inequality,
\begin{equation}\label{4.1.1}
\begin{split}
E_1\leq (N+1)\sum\limits_{n=0}^{N-1}\int_{ns}^{(n+1)s}\int_{\Omega}\sum\limits_{k=0}^{N}|\Delta_{n,k}(\hat x,x_3)|^2d\hat{x}dx_3\\
=(N+1)\sum\limits_{k=0}^{N}\sum\limits_{n=0}^{N-1}\int_{ns}^{(n+1)s}\int_{\Omega}|\Delta_{n,k}(\hat x,x_3)|^2d\hat{x}dx_3.
\end{split}
\end{equation}
Let $\frac{1}{2}<\alpha<1$ be a constant to be chosen later. For every $k$ fixed in $\{0,1,\dddot\ ,N\}$, we write
\begin{equation}\label{4.1.2}
\sum\limits_{n=0}^{N-1}\int_{ns}^{(n+1)s}\int_{\Omega}|\Delta_{n,k}(\hat x,x_3)|^2d\hat{x}dx_3=E_{1,k}+\tilde E_{1,k},
\end{equation}
where
\begin{equation*}
E_{1,k}=\sum\limits_{|n-k|\leq s^{-\alpha}}\int_{ns}^{(n+1)s}\int_{\Omega}|\Delta_{n,k}(\hat x,x_3)|^2d\hat{x}dx_3
\end{equation*}
and
\begin{equation*}
\tilde E_{1,k}=\sum\limits_{|n-k|>s^{-\alpha}}\int_{ns}^{(n+1)s}\int_{\Omega}|\Delta_{n,k}(\hat x,x_3)|^2d\hat{x}dx_3.
\end{equation*}
The sums above are taken over $n$ in the indicated subsets of $\{0,1,\dddot\ ,N-1\}$. If $|n-k|\leq s^{-\alpha}$, we have $|ns-ks|\leq s^{1-\alpha}\rightarrow 0$ as $s\rightarrow 0$. Therefore, for $s$ sufficiently small and for each $n\ne k$ in $\{0,1,\dddot\ ,N-1\}$ satisfying $|n-k|\leq s^{-\alpha}$, the following holds for any $(\hat{x},x_3)\in\Omega\times[ns,(n+1)s)$:
\begin{equation*}
\begin{split}
|\Delta_{n,k}(\hat x,x_3)|&\leq|\frac{\partial}{\partial x_1}[S_k(g_k^2)(\hat{x},x_3)]|+|\frac{\partial}{\partial x_1}[S_k(g_k^2)(\hat{x},ns)]|\\
&\leq 2[\frac{\partial}{\partial x_1}S_k(g_k^2)]^*(\hat{x},ks)
\end{split}
\end{equation*}
and thus
\begin{equation*}
\int_{ns}^{(n+1)s}\int_{\Omega}|\Delta_{n,k}(\hat x,x_3)|^2d\hat{x}dx_3\leq 4s\lVert[\frac{\partial}{\partial x_1}S_k(g_k^2)]^*(\hat x,ks)\rVert_{L^2(\Omega\times\{ks\})}^2,
\end{equation*}
where from Section 2, $[\frac{\partial}{\partial x_1}S_k(g_k^2)]^*(\hat{x},ks)$ is the nontangential maximal function of the tangential derivative $\frac{\partial}{\partial x_1}S_k(g_k^2)$ at the point $(\hat{x},ks)$ on the $k$th layer $\Omega_k$. By \eqref{nontangential_maximal} we have
\begin{equation*}
\lVert[\frac{\partial}{\partial x_1}S_k(g_k^2)]^*\rVert_{L^2(\Omega\times\{ks\})}^2\leq C\lVert g_k^2\rVert_{L^2(\Omega\times\{ks\})}^2\leq Cs^2\lVert\partial_{2}u_{k}-\imath A_k^2u_k\rVert_{L^2(\Omega)}^2,
\end{equation*}
where $C$ is a constant depending only on $R_0$. For $n=k$, we know from \eqref{grad_tn} that
$$\lVert\hat{\nabla}t_k^2\rVert_{L^2(\Omega)}\leq C\lVert g_k^2\rVert_{L^2(\Omega\times\{ks\})}.$$
Hence
\begin{equation}\label{TraceEstimateI1_1}
\begin{split}
E_{1,k}\leq&\sum\limits_{|n-k|\leq s^{-\alpha}}4sC\lVert g_k^2\rVert_{L^2(\Omega\times\{ks\})}^2\leq C\cdot s^{-\alpha}\cdot s\lVert g_k^2\rVert_{L^2(\Omega\times\{ks\})}^2\\
=&C\cdot s^{1-\alpha}\lVert g_k^2\rVert_{L^2(\Omega\times\{ks\})}^2
\end{split}
\end{equation}
for all $s$ sufficiently small and some constant $C$ depending only on $R_0$.

In order to estimate $\tilde E_{1,k}$, consider $n$ in $\{0,1,\dddot\ ,N-1\}$ such that $|n-k|>s^{-\alpha}$. Recall that $S_k(g_k^i)$ is harmonic in $\mathbb{R}^3\setminus\overline{\Omega}_k$. Without loss of generality, we may assume $k+s^{-\alpha}<n\leq N-1$. (The analysis for $0\leq n< k-s^{-\alpha}$ is similar.) Let $D_U=\{(\hat x,x_3)\in D:x_3\geq ks+s^{1-\alpha}\}$. Then it is clear that $\Omega\times[ns,(n+1)s]\subset D_U$ for every $n$ satisfying the above assumptions. Take some bounded smooth domain $D_k\subset\{(\hat x,x_3)\in\mathbb{R}^3:x_3> ks+\frac{s^{1-\alpha}}{2}\}$ in $\mathbb{R}^3$ such that $\overline{\Omega}\times\{ks+\frac{s^{1-\alpha}}{2}\}$ is a flat portion of the boundary of $D_k$, $\overline{D_U}\subset D_k$ and $dist(\overline{D_U},\partial D_k)\geq\frac{s^{1-\alpha}}{2}$. Then $S_k(g_k^i)$ is harmonic in $D_k$ and for each $x\in D_k$, it follows from \eqref{Sk} and H\"{o}lder's inequality that
\begin{equation*}
\begin{split}
|S_k(g_k^i)(x)|=&|\int_{\Omega}\frac{c}{|x-(\hat{y},ks)|}h_k^i(\hat{y})d\hat{y}|\\
\leq&\frac{1}{4\pi}\int_{\Omega}\frac{1}{|x_3-ks|}\cdot|h_k^i(\hat{y})|d\hat{y}\\
\leq&\frac{1}{4\pi}\int_{\Omega}\frac{2}{s^{1-\alpha}}\cdot|h_k^i(\hat{y})|d\hat{y}\leq\frac{C_0}{s^{1-\alpha}}|\Omega|^{\frac{1}{2}}\lVert h_k^i\rVert_{L^2(\Omega)},
\end{split}
\end{equation*}
and therefore $\sup\limits_{D_k}|S_k(g_k^i)|\leq\frac{C_0}{s^{1-\alpha}}|\Omega|^{\frac{1}{2}}\lVert h_k^i\rVert_{L^2(\Omega)}$. For $(\hat x,x_3)\in\overline{\Omega}\times[ns,(n+1)s]$, we have (since $(\hat x,x_3)\in\overline{\Omega}\times[ns,(n+1)s]\subset\subset\mathbb{R}^3\setminus\overline{\Omega}_k$ and $\Delta_{n,k}$ is harmonic in $\mathbb{R}^3\setminus\overline{\Omega}_k$)
\begin{equation*}
\begin{split}
|\Delta_{n,k}(\hat x,x_3)|\leq&\sup_{\overline{\Omega}\times[ns,(n+1)s]}|\frac{\partial^2}{\partial x_1\partial x_3}S_k(g_k^2)|\cdot|x_3-ns|\\
\leq&s\cdot\sup_{D_U}|\frac{\partial^2}{\partial x_1\partial x_3}S_k(g_k^2)|.
\end{split}
\end{equation*}
By Theorem 2.10 in \cite{Gilbarg} and the fact that $dist(\overline{D_U},\partial D_k)\geq\frac{s^{1-\alpha}}{2}$, we have
\begin{equation*}
\sup_{D_U}|\frac{\partial^2}{\partial x_1\partial x_3}S_k(g_k^2)|\leq(\frac{12}{s^{1-\alpha}})^2\sup_{D_k}|S_k(g_k^2)|\leq\frac{C_0}{s^{3-3\alpha}}|\Omega|^{\frac{1}{2}}\lVert h_k^2\rVert_{L^2(\Omega)}.
\end{equation*}
Hence we obtain
\begin{equation*}
|\Delta_{n,k}(\hat x,x_3)|\leq C_0s^{3\alpha-2}|\Omega|^{\frac{1}{2}}\lVert h_k^2\rVert_{L^2(\Omega)}
\end{equation*}
and therefore
\begin{equation*}
\int_{\Omega}|\Delta_{n,k}(\hat x,x_3)|^2d\hat{x}\leq C_0 s^{6\alpha-4}|\Omega|\cdot\lVert h_k^2\rVert_{L^2(\Omega)}^2\cdot|\Omega|=C_0 |\Omega|^2s^{6\alpha-4}\lVert h_k^2\rVert_{L^2(\Omega)}^2.
\end{equation*}
To get the best rate of convergence in $s$ as $s\rightarrow 0$, we may take $\alpha=\frac{5}{7}$ so that $1-\alpha=6\alpha-4=\frac{2}{7}$. The above estimate then becomes
\begin{equation*}
\int_{\Omega}|\Delta_{n,k}(\hat x,x_3)|^2d\hat{x}\leq C_0|\Omega|^2s^{\frac{2}{7}}\lVert h^2_k\rVert^2_{L^2(\Omega)}.
\end{equation*}

Integrating over $[ns,(n+1)s]$ and observing that the cardinality of the indices for the summation on $n$ in $\tilde E_{1,k}$ is less than $N$ and $Ns=L$ is fixed, we obtain
\begin{equation}\label{TraceEstimateI1_2}
\tilde E_{1,k}\leq L\cdot C_0|\Omega|^2s^{\frac{2}{7}}\lVert h^2_k\rVert^2_{L^2(\Omega)}\leq Cs^{\frac{2}{7}}\lVert h^2_k\rVert^2_{L^2(\Omega)}
\end{equation}
for some constant $C$ depending only on $R_0$ and $L$. Combining \eqref{4.1.1}, \eqref{4.1.2}, \eqref{TraceEstimateI1_1} and \eqref{TraceEstimateI1_2} yields
\begin{equation*}
E_1\leq(N+1)\sum\limits_{k=0}^{N}Cs^{\frac{2}{7}}\lVert h^2_k\rVert^2_{L^2(\Omega)}.
\end{equation*}
Similarly we have
\begin{equation*}
E_2\leq(N+1)\sum\limits_{k=0}^{N}Cs^{\frac{2}{7}}\lVert h^1_k\rVert^2_{L^2(\Omega)}.
\end{equation*}
\end{proof}

\begin{theorem}\label{TraceEstimate}
Under the assumptions of Lemma 4.1, we have
\begin{equation*}
\frac{1}{2}\sum\limits_{n=0}^{N-1}\int_{ns}^{(n+1)s}\int_{\Omega}|\hat{\textnormal{curl}}\hat{A}(\hat x,x_3)-\hat{\textnormal{curl}}\hat{A}_n(\hat x)|^2d\hat{x}dx_3\leq Cs^{\frac{2}{7}}\mathcal{G}_{LD}^{\epsilon,s}(\{u_n\}_{n=0}^N,\vec{A})
\end{equation*}
\textcolor{black}{for all $s$ sufficiently small (independent of $\epsilon$) and some constant $C$ depending only on $R_0$ and $L$.}
\end{theorem}
\begin{proof}
First note that
\begin{equation*}
\begin{split}
\sum\limits_{k=0}^{N}(&\lVert h^1_k\rVert^2_{L^2(\Omega)}+\lVert h^2_k\rVert^2_{L^2(\Omega)})\\
&\leq s^2\sum\limits_{k=0}^{N}(\lVert \partial_{1}u_{k}-\imath A_k^1u_k\rVert^2_{L^2(\Omega)}+\lVert \partial_{2}u_{k}-\imath A_k^2u_k\rVert^2_{L^2(\Omega)})\\
&=s^2\sum\limits_{k=0}^{N}\lVert\hat{\nabla}_{\hat{A}_k}u_k\rVert_{L^2(\Omega)}^2.
\end{split}
\end{equation*}
Since $\frac{s}{2}\sum\limits_{k=0}^{N}\lVert\hat{\nabla}_{\hat{A}_k}u_k\rVert_{L^2(\Omega)}^2$ is part of the Lawrence-Doniach energy, it is clear that
\begin{equation}\label{4.2.1}
\sum\limits_{k=0}^{N}(\lVert h^1_k\rVert^2_{L^2(\Omega)}+\lVert h^2_k\rVert^2_{L^2(\Omega)})\leq 2s\mathcal{G}_{LD}^{\epsilon,s}(\{u_n\}_{n=0}^N,\vec{A}).
\end{equation}
Hence, it follows from Lemmas 4.1, 4.2 and \eqref{4.2.1} that
\begin{equation*}
\begin{split}
\frac{1}{2}\sum\limits_{n=0}^{N-1}\int_{ns}^{(n+1)s}&\int_{\Omega}|\hat{\text{curl}}\hat{A}(\hat x,x_3)-\hat{\text{curl}}\hat{A}_n(\hat x)|^2d\hat{x}dx_3\leq E_1+E_2\\
\leq&(N+1)Cs^{\frac{2}{7}}\sum\limits_{k=0}^{N}(\lVert h^1_k\rVert^2_{L^2(\Omega)}+\lVert h^2_k\rVert^2_{L^2(\Omega)})\\
\leq&(N+1)Cs^{\frac{2}{7}}\cdot 2s\mathcal{G}_{LD}^{\epsilon,s}(\{u_n\}_{n=0}^N,\vec{A})\leq Cs^{\frac{2}{7}}\mathcal{G}_{LD}^{\epsilon,s}(\{u_n\}_{n=0}^N,\vec{A})
\end{split}
\end{equation*}
\textcolor{black}{for all $\epsilon$ sufficiently small} and some constant $C$ depending only on $R_0$ and $L$.
\end{proof}

\begin{remark}
\textcolor{black}{By a trivial modification of the proofs of Lemmas 4.1 and 4.2, it is clear that under the assumptions of Lemma 4.1, we have that for} $i \in \{1,2\}$, $j \in \{1,2\}$, and $C$ as above,
\begin{equation}\label{lem4.1.2}
\begin{split}
\sum\limits_{n=0}^{N-1}\int_{ns}^{(n+1)s}\int_{\Omega}\big|\frac{\partial}{\partial x_j}\left(A^i(\hat x,x_3)-A_n^i(\hat x)\right)\big|^2&d\hat{x}dx_3\\
\leq&(N+1)\sum\limits_{k=0}^{N}Cs^{\frac{2}{7}}\lVert h^i_k\rVert^2_{L^2(\Omega)}
\end{split}
\end{equation}
\textcolor{black}{for all $s$ sufficiently small.}  Therefore, summing in \eqref{lem4.1.2} over all $i$ and $j$ in $\{1,2\}$ and using \eqref{4.2.1}, we obtain
\begin{equation*}
\sum\limits_{n=0}^{N-1}\int_{ns}^{(n+1)s}\int_{\Omega}\bigl|\hat{\nabla}\bigr(\hat{A}(\hat x,x_3)-\hat{A}_n(\hat x)\bigr)\bigr|^2d\hat{x}dx_3\leq Cs^{\frac{2}{7}}\mathcal{G}_{LD}^{\epsilon,s}(\{u_n\}_{n=0}^N,\vec{A})
\end{equation*}
for $C$ as in Theorem 4.3.
\end{remark}

\section{Lower bound}

Theorem 4.3 provides the main step
in our proof of the lower bound on the minimal Lawrence-Doniach energy. This relies on approximating the energy of the magnetic term $|\hat{\text{curl}}\hat{A}-h_{ex}|^2$ by the sum of its traces $|\hat{\text{curl}}\hat{A}_n-h_{ex}|^2$ on the layers \textcolor{black}{$\Omega\times\{ns\}$} in the thin domains $\Omega\times[ns,(n+1)s)$. Theorem 4.3 indicates that the error from this approximation is indeed of a lower order compared to the leading order term of the total energy.

\begin{theorem}
Assume $|\ln\epsilon|\ll h_{ex}\ll\frac{1}{\epsilon^2}$ as $\epsilon\rightarrow 0$. Let $(\{u_n\}_{n=0}^N,\vec{A})\in[H^1(\Omega;\mathbb C)]^{N+1} \times K$ be a minimizer of $\mathcal{G}_{LD}^{\epsilon,s}$. Then we have
\begin{equation*}
\mathcal{G}_{LD}^{\epsilon,s}(\{u_n\}_{n=0}^N,\vec{A})\geq \frac{|D|}{2}h_{ex} (\ln\frac{1}{\epsilon\sqrt{h_{ex}}})\bigl(1-o_{\epsilon}(1)-C s^{\frac{1}{7}}\bigr)
\end{equation*}
for all $\epsilon$ and $s$ sufficiently small, where $C$ is a constant depending only on $R_0$ and $L$.
\end{theorem}

\begin{proof}
By dropping the nonnegative Josephson coupling term and the square of the $L^2$ norm of the first two components of $\nabla\times\vec A-h_{ex}\vec e_3$, it is clear that
\begin{equation*}
\begin{split}
\mathcal{G}_{LD}^{\epsilon,s}(\{u_n\}_{n=0}^N,\vec{A})&\geq s\sum^N_{n=0} \int_\Omega\left[ \frac{1}{2}|\hat{\nabla}_{\hat{A}_{n}}u_n|^2+\frac{1}{4\epsilon^2}(1-|u_n|^2)^2\right] d\hat{x}\\
&+\frac{1}{2}\int_{\mathbb{R}^3}(\hat{\text{curl}}\hat{A}-h_{ex})^{2}dx.
\end{split}
\end{equation*}
Then
\begin{equation*}
\begin{split}
\frac{1}{2}\int_{\mathbb{R}^3}\bigl( &\hat{\text{curl}}\hat{A}(x)-h_{ex}\bigr) ^{2}dx\geq\frac{1}{2}\int_{D}\bigl(\hat{\text{curl}}\hat{A}(x)-h_{ex}\bigr)^{2}dx\\
=&\frac{1}{2}\sum\limits_{n=0}^{N-1}\int_{ns}^{(n+1)s}\int_{\Omega}\bigl(\hat{\text{curl}}\hat{A}(\hat x,x_3)-h_{ex}\bigr)^{2}d\hat{x}dx_3\\
=&\frac{1}{2}\sum\limits_{n=0}^{N-1}\int_{ns}^{(n+1)s}\int_{\Omega}\left[ \hat{\text{curl}}\bigl(\hat{A}(\hat x,x_3)-\hat{A}_n(\hat x)\bigr)+\bigl(\hat{\text{curl}}\hat{A}_n(\hat x)-h_{ex}\bigr)\right] ^{2}d\hat{x}dx_3.
\end{split}
\end{equation*}
Applying the elementary inequality $(a+b)^2\geq a^2-2|a|\cdot|b|$ yields
\begin{equation}\label{LB.2}
\begin{split}
\frac{1}{2}\int_{D}(\hat{\text{curl}}\hat{A}-h_{ex})^{2}&dx\geq\frac{1}{2}\sum\limits_{n=0}^{N-1}\int_{ns}^{(n+1)s}\int_{\Omega}(\hat{\text{curl}}\hat{A}_n-h_{ex})^{2}d\hat{x}dx_3\\
&-\sum\limits_{n=0}^{N-1}\int_{ns}^{(n+1)s}\int_{\Omega}|\hat{\text{curl}}(\hat{A}-\hat{A}_n)|\cdot|\hat{\text{curl}}\hat{A}_n-h_{ex}|d\hat{x}dx_3.
\end{split}
\end{equation}
Therefore
\begin{equation}\label{LB.3}
\begin{split}
&\mathcal{G}_{LD}^{\epsilon,s}(\{u_n\}_{n=0}^N,\vec{A})\\
\geq& s\sum^{N-1}_{n=0} \int_\Omega\left[ \frac{1}{2}|\hat{\nabla}_{\hat{A}_{n}}u_n|^2+\frac{1}{4\epsilon^2}(1-|u_n|^2)^2\right] d\hat{x}+\frac{1}{2}\int_{D}(\hat{\text{curl}}\hat{A}-h_{ex})^{2}dx\\
\geq& s\sum^{N-1}_{n=0}\left\lbrace \int_\Omega\left[ \frac{1}{2}|\hat{\nabla}_{\hat{A}_{n}}u_n|^2+\frac{1}{4\epsilon^2}(1-|u_n|^2)^2\right] d\hat{x}+\frac{1}{2}\int_{\Omega}(\hat{\text{curl}}\hat{A}_n-h_{ex})^{2}d\hat{x}\right\rbrace \\
-&\sum\limits_{n=0}^{N-1}\int_{ns}^{(n+1)s}\int_{\Omega}|\hat{\text{curl}}(\hat{A}-\hat{A}_n)|\cdot|\hat{\text{curl}}\hat{A}_n-h_{ex}|d\hat{x}dx_3.
\end{split}
\end{equation}
It was proved in \cite{BK} that $u_n\in H^1(\Omega;\mathbb{C})$ and $\hat{A}_n\in H^1_{loc}(\mathbb{R}^2;\mathbb{R}^2)$ for all $n=0,1,\dddot\ ,N-1$. Therefore each pair $(u_n,\hat{A}_n)$ is in the admissible set for the minimization of the two-dimensional Ginzburg-Landau energy $F_{\epsilon}$, and by Theorem 8.1 proved by Sandier and Serfaty in \cite{SS} we have, for each $n=0,1,\dddot\ ,N-1$,
\begin{equation}\label{2DLB}
\begin{split}
\int_\Omega\bigl[\frac{1}{2}|\hat{\nabla}_{\hat{A}_{n}}u_n|^2+&\frac{1}{4\epsilon^2}(1-|u_n|^2)^2 \bigr]d\hat{x}+\frac{1}{2}\int_{\Omega}(\hat{\text{curl}}\hat{A}_n-h_{ex})^{2}d\hat{x}\\
&\geq\frac{|\Omega|}{2}h_{ex}\ln\frac{1}{\epsilon\sqrt{h_{ex}}}(1-o_{\epsilon}(1))
\end{split}
\end{equation}
as $\epsilon\rightarrow 0$. \textcolor{black}{(The proof in \cite{SS} is for smooth domains, but the estimate
\eqref{2DLB} follows by approximating the Lipschitz domain $\Omega$ with an appropriate sequence of smooth domains.)}  Using the Cauchy-Schwartz inequality and H\"{o}lder's inequality, we obtain
\begin{equation}\label{LB.4}
\begin{split}
&\sum\limits_{n=0}^{N-1}\int_{ns}^{(n+1)s}\int_{\Omega}|\hat{\text{curl}}(\hat{A}-\hat{A}_n)|\cdot|\hat{\text{curl}}\hat{A}_n-h_{ex}|d\hat{x}dx_3\\
\leq
&\left( \sum\limits_{n=0}^{N-1}\int_{ns}^{(n+1)s}\int_{\Omega}|\hat{\text{curl}}(\hat{A}-\hat{A}_n)|^2d\hat{x}dx_3\right) ^{\frac{1}{2}}\\
&\ \ \ \ \ \ \ \ \ \ \ \ \ \ \ \ \ \ \ {\cdot} \left( \sum\limits_{n=0}^{N-1}\int_{ns}^{(n+1)s}\int_{\Omega}|\hat{\text{curl}}\hat{A}_n-h_{ex}|^2d\hat{x}dx_3\right)
^{\frac{1}{2}}.
\end{split}
\end{equation}
Since $\hat{\text{curl}}\hat{A}_n-h_{ex}=(\hat{\text{curl}}(\hat{A}_n-\hat{A}))+(\hat{\text{curl}}\hat{A}-h_{ex})$, using the elementary inequality $(a+b)^2\leq 2a^2+2b^2$ yields
\begin{equation*}
\begin{split}
&\sum\limits_{n=0}^{N-1}\int_{ns}^{(n+1)s}\int_{\Omega}|\hat{\text{curl}}\hat{A}_n-h_{ex}|^2d\hat{x}dx_3\\
\leq&2\sum\limits_{n=0}^{N-1}\int_{ns}^{(n+1)s}\int_{\Omega}|\hat{\text{curl}}(\hat{A}_n-\hat{A})|^2d\hat{x}dx_3+2\sum\limits_{n=0}^{N-1}\int_{ns}^{(n+1)s}\int_{\Omega}|\hat{\text{curl}}\hat{A}-h_{ex}|^2d\hat{x}dx_3.
\end{split}
\end{equation*}
As a result of this, Theorem \ref{TraceEstimate} and Theorem \ref{UB}, we have
\begin{equation}\label{curlAn}
\sum\limits_{n=0}^{N-1}\int_{ns}^{(n+1)s}\int_{\Omega}|\hat{\text{curl}}\hat{A}_n-h_{ex}|^2d\hat{x}dx_3\leq o_{\epsilon,s}(1)M_{\epsilon}+4\cdot M_{\epsilon}(1+o_{\epsilon,s}(1))\leq C_0\cdot M_{\epsilon}
\end{equation}
\textcolor{black}{for all $\epsilon$ and $s$ sufficiently small.}  By \eqref{LB.4}, Theorem 4.3 and \eqref{curlAn}, we obtain
\begin{equation}\label{LB_CrossProduct}
\begin{split}
\sum\limits_{n=0}^{N-1}\int_{ns}^{(n+1)s}&\int_{\Omega}|\hat{\text{curl}}(\hat{A}-\hat{A}_n)|\cdot|\hat{\text{curl}}\hat{A}_n-h_{ex}|d\hat{x}dx_3\\
\leq&(Cs^{\frac{2}{7}}M_{\epsilon})^{\frac{1}{2}}\cdot(C_0M_{\epsilon})^{\frac{1}{2}}\leq Cs^{\frac{1}{7}}M_{\epsilon}
\end{split}
\end{equation}
for some \textcolor{black}{constants} $C$ depending only on $R_0$ and $L$ for all $\epsilon$ and $s$ sufficiently small. By \eqref{LB.2}, \eqref{2DLB}, \eqref{LB_CrossProduct} and the definition of $M_{\epsilon}$, i.e., $M_{\epsilon}=\frac{|D|}{2}h_{ex}\ln\frac{1}{\epsilon\sqrt{h_{ex}}}$, we conclude that
\begin{equation*}
\begin{split}
\mathcal{G}_{LD}^{\epsilon,s}(\{u_n\}_{n=0}^N,\vec{A})&\geq sN\cdot\frac{|\Omega|}{2}h_{ex}\ln\frac{1}{\epsilon\sqrt{h_{ex}}}(1-o_{\epsilon}(1))-Cs^{\frac{1}{7}}\frac{|D|}{2}h_{ex}\ln\frac{1}{\epsilon\sqrt{h_{ex}}}\\
&=\frac{|D|}{2}h_{ex}\ln\frac{1}{\epsilon\sqrt{h_{ex}}}(1-o_{\epsilon}(1))-Cs^{\frac{1}{7}}\frac{|D|}{2}h_{ex}\ln\frac{1}{\epsilon\sqrt{h_{ex}}}\\
&=\frac{|D|}{2}h_{ex}\ln\frac{1}{\epsilon\sqrt{h_{ex}}}(1-o_{\epsilon}(1)-Cs^{\frac{1}{7}})
\end{split}
\end{equation*}
\textcolor{black}{for all $\epsilon$ and $s$ sufficiently small and} some constant $C$ depending only on $R_0$ and $L$. This proves the theorem.
\end{proof}

By combining Theorems 3.1 and 5.1 we obtain Theorem 1. Next we prove Theorem 2:

\begin{proof}[Proof of Theorem 2]
By \eqref{LB.3}, \eqref{2DLB},
and \eqref{LB_CrossProduct}, we see that the leading term in our lower bound of the minimum Lawrence-Doniach energy comes from two terms, since
\begin{equation}\label{Thm2.1}
\begin{split}
\mathcal{G}_{LD}^{\epsilon,s}(\{u_n\}_{n=0}^N,\vec{A})\geq& s\sum^N_{n=0} \int_\Omega\left[ \frac{1}{2}|\hat{\nabla}_{\hat{A}_{n}}u_n|^2+\frac{1}{4\epsilon^2}(1-|u_n|^2)^2\right] d\hat{x}\\
+&\frac{1}{2}\int_{D}(\hat{\text{curl}}\hat{A}-h_{ex})^{2}dx\\
\geq&\frac{|D|}{2}h_{ex}\ln\frac{1}{\epsilon\sqrt{h_{ex}}}(1-o_{\epsilon}(1)-Cs^{\frac{1}{7}})
\end{split}
\end{equation}
for some constant $C$ depending only on $R_0$ and $L$ and for all $\epsilon$ and $s$ sufficiently small. As a result of \eqref{Thm2.1} and Theorem 3.1, we conclude that
\begin{equation*}
\begin{split}
&s\sum^{N-1}_{n=0}\int_\Omega\frac{1}{2\lambda^{2}s^2}|u_{n+1}-u_{n}e^{\imath \int_{ns}^{(n+1)s}A^{3}dx_{3}}|^{2}d\hat{x}+\frac{1}{2}\int_{\mathbb{R}^3\setminus D}|\nabla\times\vec{A}-h_{ex}\vec{e}_3|^{2}dx\\
&\qquad\qquad+\frac{1}{2}\int_{D}[(\frac{\partial A^3}{\partial x_2}-\frac{\partial A^2}{\partial x_3})^{2}+(\frac{\partial A^1}{\partial x_3}-\frac{\partial A^3}{\partial x_1})^{2}]dx\leq o_{\epsilon,s}(1)M_{\epsilon}.
\end{split}
\end{equation*}
This proves Theorem 2.
\end{proof}

\begin{corollary}
Under the assumptions of Theorem 2, we have
\begin{equation*}
\bigl| \mathcal{G}_{LD}^{\epsilon, s}(\{u_n\}_{n=0}^N, \vec{A})-\sum^{N-1}_{n=0}sF_\epsilon(u_n,\hat {A}_n)\bigr|  \leq o_{\epsilon,s}(1) M_{\epsilon}
\end{equation*}
\textcolor{black}{for all $\epsilon$ and $s$ sufficiently small.}
\end{corollary}

\begin{proof}
By \eqref{LB.3} and \eqref{LB_CrossProduct}, we have
\begin{equation*}
\textcolor{black}{\mathcal{G}_{LD}^{\epsilon, s}(\{u_n\}_{n=0}^N, \vec{A})\geq \big(\sum^{N-1}_{n=0}sF_\epsilon(u_n,\hat {A}_n)\big)
-Cs^{1/7} M_{\epsilon}}
\end{equation*}
for all $\epsilon$ and $s$ sufficiently small where $C$ depends only on $R_0$ and $L$. By the Cauchy-Schwartz inequality, Theorem 4.3, Theorem 3.1 and \eqref{curlAn}, we have
\begin{equation*}
\left|\frac{1}{2}\int_{D}|\hat{\text{curl}}\hat A-h_{ex}|^2 dx-\sum\limits_{n=0}^{N-1}\frac{s}{2}\int_{\Omega}|\hat{\text{curl}}\hat{A}_n-h_{ex}|^2 d\hat x\right|\leq Cs^{\frac{1}{7}}M_{\epsilon}
\end{equation*}
for some constant $C$ depending only on $R_0$ and $L$ and all $\epsilon$ and $s$ sufficiently small. Combining this with Theorem 2 and the definition of $\mathcal{G}_{LD}^{\epsilon, s}$, we obtain
\begin{equation*}
\bigl| \mathcal{G}_{LD}^{\epsilon, s}(\{u_n\}_{n=0}^N, \vec{A})-\sum^{N-1}_{n=0}sF_\epsilon(u_n,\hat {A}_n)\bigr|  \leq [Cs^{\frac{1}{7}} + o_{\epsilon}(1)] M_{\epsilon}.
\end{equation*}
\end{proof}

Finally we prove Theorem 3.

\begin{theorem}\label{t5.3}
Under the assumptions of Theorem 1, we have
\begin{equation}\label{5.3.1}
\textcolor{black}{\frac{1}{N+1}\sum\limits_{n=0}^{N}\int_{\Omega} (1-|u_n|^2)^2 d\hat{x} +
\textcolor{black}{\frac{1}{2}}\int_{\mathbb{R}^3} |\frac{\nabla\times\vec{A}}{h_{ex}}-\vec{e}_3|^2 dx
\rightarrow 0}
\end{equation}
and
\begin{equation}\label{5.3.2}
\frac{1}{N+1}\sum_{n=0}^N\frac{\mu_n}{h_{ex}}\rightarrow d\hat{x}\ \text{in}\ H^{-1}(\Omega)
\end{equation}
as $(\epsilon,s)\rightarrow(0,0)$, where \textcolor{black}{$d\hat{x}$ is the two-dimensional Lebesgue measure and} $\mu_n$ is the vorticity on the $n$th layer defined as $$\mu_n=\hat{\textnormal{curl}}(\imath u_n,\hat{\nabla}_{\hat{A}_n}u_n)+\hat{\textnormal{curl}}\hat{A}_n.$$
\end{theorem}
\begin{proof}
It follows immediately from Theorem 3.1 that the first integral in \eqref{5.3.1}
is bounded above by \textcolor{black}{$\frac{4}{L} \epsilon^2 M_{\epsilon}[1+\frac {s}{L}+o_{\epsilon}(1)]$ and the second integral is bounded above by \textcolor{black}{$|D| \frac{1}{h_{ex}} \ln(\frac{1}{\epsilon \sqrt{h_{ex}}})[1+\frac {s}{L}+o_{\epsilon}(1)]$}.} By our assumption,
$|\ln\epsilon|\ll h_{ex}\ll\frac{1}{\epsilon^2}$ and recall that $\frac {s}{L} \leq 1$.  Hence both of these terms converge to zero as $\epsilon$ tends to zero.


By the regularity results proved by Bauman and Ko in \cite{BK}, we know that $(\imath u_n,\hat{\nabla}_{\hat{A}_n}u_n)\in L^2(\Omega;\mathbb{R}^2)$ and $\hat{A}_n\in H^1(\Omega;\mathbb{R}^2)$. Thus
\begin{equation*}
\mu_n=\hat{\textnormal{curl}}(\imath u_n,\hat{\nabla}_{\hat{A}_n}u_n)+\hat{\textnormal{curl}}\hat{A}_n \in H^{-1}(\Omega),
\end{equation*}
and since $|u_n|\leq 1$, we have
\begin{equation*}
\lVert\mu_n-h_{ex}\rVert^2_{H^{-1}(\Omega)}\leq\lVert\hat{\nabla}_{\hat{A}_n}u_n\rVert^2_{L^{2}(\Omega)}+\lVert\hat{\text{curl}}\hat{A}_n-h_{ex}\rVert^2_{L^{2}(\Omega)}.
\end{equation*}
Therefore, by Theorem \ref{UB} and \eqref{curlAn}, we get
\begin{equation*}
\begin{split}
s\sum\limits_{n=0}^{N}&\lVert\mu_n-h_{ex}\rVert^2_{H^{-1}(\Omega)}\leq s\sum\limits_{n=0}^{N}\lVert\hat{\nabla}_{\hat{A}_n}u_n\rVert^2_{L^{2}(\Omega)}+s\sum\limits_{n=0}^{N}\lVert\hat{\text{curl}}\hat{A}_n-h_{ex}\rVert^2_{L^{2}(\Omega)}\\
&=s\sum\limits_{n=0}^{N}\lVert\hat{\nabla}_{\hat{A}_n}u_n\rVert^2_{L^{2}(\Omega)}+\sum\limits_{n=0}^{N}\int_{ns}^{(n+1)s}\int_{\Omega}|\hat{\text{curl}}\hat{A}_n-h_{ex}|^2d\hat{x}dx_3\\
&\leq 2M_{\epsilon}(1+o_{\epsilon,s}(1))+C_0M_{\epsilon}\leq C_0M_{\epsilon}.
\end{split}
\end{equation*}
This implies that
\begin{equation}\label{mu}
s\sum\limits_{n=0}^{N}\lVert\frac{\mu_n}{h_{ex}}-d\hat{x}\rVert^2_{H^{-1}(\Omega)}\leq\frac{C_0 M_{\epsilon}}{h_{ex}^2}\rightarrow 0
\end{equation}
as $(\epsilon,s)\rightarrow (0,0)$. Note that
\begin{equation*}
\begin{split}
\lVert\frac{1}{N+1}\sum\limits_{n=0}^{N}\frac{\mu_n}{h_{ex}}-d\hat{x}\rVert_{H^{-1}(\Omega)}^2=&\lVert\frac{1}{N+1}\sum\limits_{n=0}^{N}(\frac{\mu_n}{h_{ex}}-d\hat{x})\rVert_{H^{-1}(\Omega)}^2\\
=&\frac{1}{(N+1)^2}\lVert\sum\limits_{n=0}^{N}(\frac{\mu_n}{h_{ex}}-d\hat{x})\rVert_{H^{-1}(\Omega)}^2.
\end{split}
\end{equation*}
By the Cauchy-Schwartz inequality and \eqref{mu}, we obtain
\begin{equation*}
\begin{split}
\lVert\frac{1}{N+1}\sum\limits_{n=0}^{N}\frac{\mu_n}{h_{ex}}-d\hat{x}\rVert_{H^{-1}(\Omega)}^2&\leq\frac{N+1}{(N+1)^2}\sum\limits_{n=0}^{N}\lVert\frac{\mu_n}{h_{ex}}-d\hat{x}\rVert_{H^{-1}(\Omega)}^2\\
&=\frac{1}{s(N+1)}\cdot s\sum\limits_{n=0}^{N}\lVert\frac{\mu_n}{h_{ex}}-d\hat{x}\rVert_{H^{-1}(\Omega)}^2\rightarrow 0
\end{split}
\end{equation*}
as $(\epsilon,s)\rightarrow (0,0)$, since $sN=L$ is the height of the domain $D$ which is fixed. This proves \eqref{5.3.2}.
\end{proof}

\section{Comparison results}

\textcolor{black}{In this section we give the proof of Corollary 1. It should be pointed out again that a result similar to the estimate (\ref{thm3.2}) has been obtained by Kachmar \cite{Ka} for the standard three-dimensional Ginzburg-Landau model. In particular, the minimum energy for the standard three-dimensional Ginzburg-Landau model with $h_{ex}$ in the regime as in Theorem 1 follows the same asymptotic formula as obtained in Theorem 1. By analogy, it should not be surprising that the estimate (\ref{thm3.2}) also holds for the three-dimensional anisotropic Ginzburg-Landau model, and hence (\ref{thm3.1}) follows naturally. However, the main point of this section is to use our results on the Lawrence-Doniach model and the connections between this model and the three-dimensional anisotropic Ginzburg-Landau model to give an alternate approach to proving the asymptotic formula for the minimum energy of the latter. Namely, we first derive the comparison result (\ref{thm3.1}) using our Theorem 1 and a result of Chapman, Du and Gunzburger on the connections between the Lawrence-Doniach model and the three-dimensional anisotropic Ginzburg-Landau model. As a result, the estimate (\ref{thm3.2}) follows immediately from (\ref{thm3.1}) and Theorem 1. Our \textcolor{black}{Corollary \ref{c1} thus} generalizes the result in \cite{Ka} to the anisotropic case through a different approach. }

Recall the definition of the anisotropic Ginzburg-Landau energy $\mathcal{G}_{AGL}^{\epsilon,\lambda}$ given in \eqref{AGL} in the introduction. Throughout this section, we set \textcolor{black}{$\mathcal{G}_{LD}^{\epsilon,s,\lambda}=
\mathcal{G}_{LD}^{\epsilon,s}$} as in the introduction. Direct calculations show that $\mathcal{G}_{AGL}^{\epsilon,\lambda}$ is invariant under the gauge transformation
\begin{equation*}
\begin{cases}
\xi(x)=\psi(x)e^{\imath g(x)} \text{ in } \Omega,\\
\vec{B}=\vec{A}+\nabla g \text{ in } \mathbb{R}^3
\end{cases}
\end{equation*}
for some $g\in H^2_{loc}(\mathbb{R}^3)$. Recall the rescaling formulas \eqref{AGLRescaling} for the anisotropic Ginzburg-Landau energies, from which we may translate estimates from \cite{Chapman} to our scaling. As pointed out in \cite{BK}, every minimizer $(\psi,\vec{A})\in H^1(D;\mathbb{C})\times E$ of $\mathcal{G}_{AGL}^{\epsilon,\lambda}$ is gauge equivalent to another pair in $H^1(D;\mathbb{C})\times K$, where the spaces $E$ and $K$ are defined in \eqref{E} and \eqref{K} respectively. The space $H^1(D;\mathbb{C})\times K$ fixes a ``Coulomb gauge" for $(\psi,\vec{A})$ as in the study of the Lawrence-Doniach energy.

\textcolor{black}{By Corollary 5.6 proved by Chapman, Du, and Gunzburger in \cite{Chapman}, for any fixed $\kappa$, $\lambda$ and $\tau > 0$
and any fixed applied magnetic field $\vec{H}=\tau \vec{e}_3$, letting $(\{\phi_n^s\}_{n=0}^N, \vec{V}_{ld}^s)$ and $(\psi_{em},\vec{A}_{em})$ be minimizers of ${\mathcal{G}}^{\kappa,s,\lambda}_{ld}|_{\vec{H}=\tau \vec {e}_3}$ and ${\mathcal{G}}^{\kappa,\lambda}_{em}|_{\vec{H}=\tau \vec {e}_3}$ respectively, we have
\begin{equation}\label{6.1}
\lim\limits_{s\rightarrow 0}{\mathcal {G}}^{\kappa,s,\lambda}_{ld}(\{\phi_n^s\}_{n=0}^N, \vec{V}_{ld}^s)|_{\vec{H}=\tau \vec {e}_3}
={\mathcal{G}}^{\kappa,\lambda}_{em}(\psi_{em},\vec{A}_{em})|_{\vec{H}=\tau \vec {e}_3}.
\end{equation}
Using the rescaling formulas \eqref{LDRescaling} and \eqref{AGLRescaling} \textcolor{black}{for $\mathcal {G}^{\kappa,s,\lambda}_{ld}$ and $\mathcal{G}_{LD}^{\epsilon,s,\lambda}$ and for $\mathcal{G}^{\kappa,\lambda}_{em}$ and $\mathcal{G}_{AGL}^{\epsilon,\lambda}$, respectively}, the relation $\kappa=\frac{1}{\epsilon}$, and \eqref{6.1} with $\tau=\frac{1}{\kappa} h_{ex}=\epsilon h_{ex}$, we obtain
\begin{equation*}
\lim\limits_{s\rightarrow 0}2\epsilon^2\mathcal{G}_{LD}^{\epsilon,s,\lambda}(\{u_n^s\}_{n=0}^N, \vec{A}^s)|_{\vec{H}=\tau \kappa \vec {e}_3}
=2\epsilon^2\mathcal{G}_{AGL}^{\epsilon,\lambda}(\zeta,\vec{B})|_{\vec{H}=\tau \kappa \vec {e}_3},
\end{equation*}
and hence
\begin{equation}\label{6.2}
\lim\limits_{s\rightarrow 0}\mathcal{G}_{LD}^{\epsilon,s,\lambda}(\{u_n^s\}_{n=0}^N, \vec{A}^s)=\mathcal{G}_{AGL}^{\epsilon,\lambda}(\zeta,\vec{B})
\end{equation}
for any fixed $\epsilon$, $\lambda$, and $h_{ex} > 0$, where $(\{u_n^s\}_{n=0}^N, \vec{A}^s)$ and $(\zeta,\vec{B})$ are minimizers of $\mathcal{G}_{LD}^{\epsilon,s,\lambda}$ and $\mathcal{G}_{AGL}^{\epsilon,\lambda}$ in an applied magnetic field $\vec{H}=h_{ex} \vec {e}_3$, respectively. Using the convergence result \eqref{6.2} and Theorem 1, we now
show:}

\begin{lemma}\label{l6.1}
\textcolor{black}{Assume $|\ln\epsilon|\ll h_{ex}\ll\epsilon^{-2}$ as $\epsilon\rightarrow 0$ and $\lambda>0$. Then
\begin{equation*}
|\min\mathcal{G}_{LD}^{\epsilon,s,\lambda}-
\min\mathcal{G}_{AGL}^{\epsilon,\lambda}|
\leq o_{\epsilon,s}(1) M_{\epsilon}
\end{equation*}
as $(\epsilon,s) \to (0,0)$, \textcolor{black}{where recall that $M_{\epsilon}=\frac{|D|}{2}h_{ex}\ln\frac{1}{\epsilon\sqrt{h_{ex}}}$}.}
\end{lemma}
\begin{proof}
\textcolor{black}{If not, for some $\lambda>0$, there exists a sequence $\{(\epsilon_n,s_n)\}$ in ${\mathbb{R}}^+ \times {\mathbb{R}}^+$ converging to $(0,0)$ and a constant $\eta_0>0$ such that
\begin{equation*}
|\min\mathcal{G}_{LD}^{\epsilon_n,s_n,\lambda}-
\min\mathcal{G}_{AGL}^{\epsilon_n,\lambda}| \geq 2\eta_0 M_{\epsilon_n}
\end{equation*}
for all $n$.  By (6.2), for each n there exists a positive constant ${\tilde {s}}_n<\epsilon_n$ such that
\begin{equation*}
|\min\mathcal{G}_{LD}^{\epsilon_n,{\tilde {s}}_n,\lambda}-
\min\mathcal{G}_{AGL}^{\epsilon_n,\lambda}| < \epsilon_n M_{\epsilon_n}
\end{equation*}
and hence
\begin{equation*}
|\min\mathcal{G}_{LD}^{\epsilon_n,{\tilde {s}}_n,\lambda}
-\min\mathcal{G}_{LD}^{\epsilon_n,s_n,\lambda} | \geq \eta_0 M_{\epsilon_n}
\end{equation*}
for all $n$ sufficiently large. Since $\{(\epsilon_n,s_n)\}$ and
$\{(\epsilon_n,{\tilde {s}}_n)\}$ both converge to $(0,0)$ as $n \to \infty$, this contradicts Theorem 1.}
\end{proof}
\textcolor{black}{Combining \textcolor{black}{Lemma} 6.1 and Theorem 1, we obtain:}

\begin{corollary}
\textcolor{black}{Assume $|\ln\epsilon|\ll h_{ex}\ll\epsilon^{-2}$ as $\epsilon\rightarrow 0$ and $\lambda>0$. Let $(\zeta_{\epsilon},\vec{B}_{\epsilon})\in H^1(D;\mathbb C)\times E$ be a minimizer of $\mathcal{G}_{AGL}^{\epsilon,\lambda}$. We have
\begin{equation*}
|\mathcal{G}_{AGL}^{\epsilon,\lambda}(\zeta_{\epsilon},\vec{B}_{\epsilon})-\frac{|D|}{2}h_{ex}\ln\frac{1}{\epsilon\sqrt{h_{ex}}}| \leq o_{\epsilon}(1)\frac{|D|}{2}h_{ex}\ln\frac{1}{\epsilon\sqrt{h_{ex}}}
\end{equation*}
as $\epsilon\rightarrow 0$.}
\end{corollary}

\textcolor{black}{Corollary 1 now follows from \textcolor{black}{Lemma \ref{l6.1}} and Corollary 6.2.}


\begin{bibdiv}
	\begin{biblist}
		
	
	

\bib{ABB1}{article}{
	author={Alama, S.},
	author={Berlinsky, A. J.},
	author={Bronsard, L.},
	title={Minimizers of the Lawrence-Doniach energy in the small-coupling
		limit: finite width samples in a parallel field},
	language={English, with English and French summaries},
	journal={Ann. Inst. H. Poincar\'e Anal. Non Lin\'eaire},
	volume={19},
	date={2002},
	number={3},
	pages={281--312},
	issn={0294-1449},
	review={\MR{1956952}},
	doi={10.1016/S0294-1449(01)00081-6},
}




\bib{ABB2}{article}{
	author={Alama, S.},
	author={Bronsard, L.},
	author={Berlinsky, A. J.},
	title={Periodic vortex lattices for the Lawrence-Doniach model of layered
		superconductors in a parallel field},
	journal={Commun. Contemp. Math.},
	volume={3},
	date={2001},
	number={3},
	pages={457--494},
	issn={0219-1997},
	review={\MR{1849651}},
	doi={10.1142/S0219199701000457},
}




\bib{ABS3}{article}{
	author={Alama, S.},
	author={Bronsard, L.},
	author={Sandier, E.},
	title={On the shape of interlayer vortices in the Lawrence-Doniach model},
	journal={Trans. Amer. Math. Soc.},
	volume={360},
	date={2008},
	number={1},
	pages={1--34},
	issn={0002-9947},
	review={\MR{2341992}},
	doi={10.1090/S0002-9947-07-04188-8},
}


\bib{ABS}{article}{
	author={Alama, S.},
	author={Bronsard, L.},
	author={Sandier, E.},
	title={On the Lawrence-Doniach model of superconductivity: magnetic
		fields parallel to the axes},
	journal={J. Eur. Math. Soc. (JEMS)},
	volume={14},
	date={2012},
	number={6},
	pages={1825--1857},
	issn={1435-9855},
	review={\MR{2984589}},
	doi={10.4171/JEMS/348},
}



\bib{ABS2}{article}{
	author={Alama, S.},
	author={Bronsard, L.},
	author={Sandier, E.},
	title={Minimizers of the Lawrence-Doniach functional with oblique
		magnetic fields},
	journal={Comm. Math. Phys.},
	volume={310},
	date={2012},
	number={1},
	pages={237--266},
	issn={0010-3616},
	review={\MR{2885619}},
	doi={10.1007/s00220-011-1399-2},
}



\bib{ABO}{article}{
	author={Alberti, G.},
	author={Baldo, S.},
	author={Orlandi, G.},
	title={Variational convergence for functionals of Ginzburg-Landau type},
	journal={Indiana Univ. Math. J.},
	volume={54},
	date={2005},
	number={5},
	pages={1411--1472},
	issn={0022-2518},
	review={\MR{2177107}},
	doi={10.1512/iumj.2005.54.2601},
}



\bib{BJOS1}{article}{
	author={Baldo, S.},
	author={Jerrard, R. L.},
	author={Orlandi, G.},
	author={Soner, H. M.},
	title={Convergence of Ginzburg-Landau functionals in three-dimensional
		superconductivity},
	journal={Arch. Ration. Mech. Anal.},
	volume={205},
	date={2012},
	number={3},
	pages={699--752},
	issn={0003-9527},
	review={\MR{2960031}},
	doi={10.1007/s00205-012-0527-2},
}



\bib{BJOS2}{article}{
	author={Baldo, S.},
	author={Jerrard, R. L.},
	author={Orlandi, G.},
	author={Soner, H. M.},
	title={Vortex density models for superconductivity and superfluidity},
	journal={Comm. Math. Phys.},
	volume={318},
	date={2013},
	number={1},
	pages={131--171},
	issn={0010-3616},
	review={\MR{3017066}},
	doi={10.1007/s00220-012-1629-2},
}



\bib{BK}{article}{
	author={Bauman, P.},
	author={Ko, Y.},
	title={Analysis of solutions to the Lawrence-Doniach system for layered
		superconductors},
	journal={SIAM J. Math. Anal.},
	volume={37},
	date={2005},
	number={3},
	pages={914--940},
	issn={0036-1410},
	review={\MR{2191782}},
	doi={10.1137/S0036141004444597},
}


\bib{Chapman}{article}{
	author={Chapman, S. J.},
	author={Du, Q.},
	author={Gunzburger, M. D.},
	title={On the Lawrence-Doniach and anisotropic Ginzburg-Landau models for
		layered superconductors},
	journal={SIAM J. Appl. Math.},
	volume={55},
	date={1995},
	number={1},
	pages={156--174},
	issn={0036-1399},
	review={\MR{1313011}},
	doi={10.1137/S0036139993256837},
}


\bib{Fabes}{article}{
	author={Fabes, E. B.},
	author={Jodeit, M., Jr.},
	author={Rivi\`ere, N. M.},
	title={Potential techniques for boundary value problems on
		$C\sp{1}$-domains},
	journal={Acta Math.},
	volume={141},
	date={1978},
	number={3-4},
	pages={165--186},
	issn={0001-5962},
	review={\MR{501367}},
	doi={10.1007/BF02545747},
}


\bib{Gilbarg}{book}{
	author={Gilbarg, D.},
	author={Trudinger, N. S.},
	title={Elliptic partial differential equations of second order},
	series={Classics in Mathematics},
	note={Reprint of the 1998 edition},
	publisher={Springer-Verlag, Berlin},
	date={2001},
	pages={xiv+517},
	isbn={3-540-41160-7},
	review={\MR{1814364}},
}


\bib{Giorgi-Phillips}{article}{
	author={Giorgi, T.},
	author={Phillips, D.},
	title={The breakdown of superconductivity due to strong fields for the
		Ginzburg-Landau model},
	note={Reprinted from SIAM J. Math. Anal. {\bf 30} (1999), no. 2, 341--359
		[MR 2002b:35235]},
	journal={SIAM Rev.},
	volume={44},
	date={2002},
	number={2},
	pages={237--256},
	issn={0036-1445},
	review={\MR{1926099}},
	doi={10.1137/S003614450139951},
}


\bib{Iye}{article}{
	author={Iye, Y.},
	title={How anisotropic are the cuprate high $T_c$ superconductors?},
	journal={Comments Cond. Mat. Phys.},
	volume={16},
	date={1992},
	pages={89--111},
}


\bib{JS}{article}{
	author={Jerrard, R. L.},
	author={Soner, H. M.},
	title={Limiting behavior of the Ginzburg-Landau functional},
	journal={J. Funct. Anal.},
	volume={192},
	date={2002},
	number={2},
	pages={524--561},
	issn={0022-1236},
	review={\MR{1923413}},
	doi={10.1006/jfan.2001.3906},
}


\bib{Ka}{article}{
	author={Kachmar, A.},
	title={The ground state energy of the three-dimensional Ginzburg-Landau
		model in the mixed phase},
	journal={J. Funct. Anal.},
	volume={261},
	date={2011},
	number={11},
	pages={3328--3344},
	issn={0022-1236},
	review={\MR{2836000}},
	doi={10.1016/j.jfa.2011.08.002},
}


\bib{Pe}{article}{
	author={Peng, G.},
	title={Convergence of the Lawrence-Doniach energy for layered
		superconductors with magnetic fields near $H_{c_1}$},
	journal={SIAM J. Math. Anal.},
	volume={49},
	date={2017},
	number={2},
	pages={1225--1266},
	issn={0036-1410},
	review={\MR{3631388}},
	doi={10.1137/16M1064398},
}


\bib{SS1}{article}{
	author={Sandier, E.},
	author={Serfaty, S.},
	title={On the energy of type-II superconductors in the mixed phase},
	journal={Rev. Math. Phys.},
	volume={12},
	date={2000},
	number={9},
	pages={1219--1257},
	issn={0129-055X},
	review={\MR{1794239}},
	doi={10.1142/S0129055X00000411},
}


\bib{SS3}{article}{
	author={Sandier, E.},
	author={Serfaty, S.},
	title={A rigorous derivation of a free-boundary problem arising in
		superconductivity},
	language={English, with English and French summaries},
	journal={Ann. Sci. \'Ecole Norm. Sup. (4)},
	volume={33},
	date={2000},
	number={4},
	pages={561--592},
	issn={0012-9593},
	review={\MR{1832824}},
	doi={10.1016/S0012-9593(00)00122-1},
}



\bib{SS4}{article}{
	author={Sandier, E.},
	author={Serfaty, S.},
	title={The decrease of bulk-superconductivity close to the second
		critical field in the Ginzburg-Landau model},
	journal={SIAM J. Math. Anal.},
	volume={34},
	date={2003},
	number={4},
	pages={939--956},
	issn={0036-1410},
	review={\MR{1969609}},
	doi={10.1137/S0036141002406084},
}


\bib{SS}{book}{
	author={Sandier, E.},
	author={Serfaty, S.},
	title={Vortices in the magnetic Ginzburg-Landau model},
	series={Progress in Nonlinear Differential Equations and their
		Applications},
	volume={70},
	publisher={Birkh\"auser Boston, Inc., Boston, MA},
	date={2007},
	pages={xii+322},
	isbn={978-0-8176-4316-4},
	isbn={0-8176-4316-8},
	review={\MR{2279839}},
}


\bib{Verchota}{article}{
	author={Verchota, G.},
	title={Layer potentials and regularity for the Dirichlet problem for
		Laplace's equation in Lipschitz domains},
	journal={J. Funct. Anal.},
	volume={59},
	date={1984},
	number={3},
	pages={572--611},
	issn={0022-1236},
	review={\MR{769382}},
	doi={10.1016/0022-1236(84)90066-1},
}


\end{biblist}
\end{bibdiv}
\end{document}